\documentclass{amsart}

\usepackage{amsmath}
\usepackage{amsfonts}
\usepackage{amssymb,enumerate}
\usepackage{amsthm}
\usepackage[all]{xy}
\usepackage{hyperref}

\theoremstyle{plain}
\newtheorem{lem}{Lemma}[section]
\newtheorem{cor}[lem]{Corollary}
\newtheorem{prop}[lem]{Proposition}
\newtheorem{thm}[lem]{Theorem}

\newtheorem{intthm}{Theorem}

\theoremstyle{definition}
\newtheorem{defn}[lem]{Definition}
\newtheorem{ex}[lem]{Example}
\newtheorem{question}[lem]{Question}

\newtheorem{disc}[lem]{Remark}

\newtheorem{fact}[lem]{Fact}

\newcommand{\cat}[1]{\mathcal{#1}}

\newcommand{\catc}{\cat{C}}

\newcommand{\Ht}{\operatorname{ht}}	
	
\newcommand{\depth}{\operatorname{depth}}

\newcommand{\ann}{\operatorname{Ann}}

\newcommand{\len}{\operatorname{len}}

\newcommand{\ext}{\operatorname{Ext}}

\newcommand{\HH}{\operatorname{H}}
\newcommand{\Hom}{\operatorname{Hom}}	
\newcommand{\coker}{\operatorname{Coker}}
\newcommand{\spec}{\operatorname{Spec}}

\newcommand{\im}{\operatorname{Im}}

\newcommand{\ideal}[1]{\mathfrak{#1}}
\newcommand{\m}{\ideal{m}}
\newcommand{\n}{\ideal{n}}
\newcommand{\p}{\ideal{p}}

\newcommand{\fm}{\ideal{m}}

\newcommand{\fa}{\ideal{a}}

\newcommand{\comp}[1]{\widehat{#1}}

\newcommand{\ass}{\operatorname{Ass}}
\newcommand{\att}{\operatorname{Att}}
\newcommand{\supp}{\operatorname{Supp}}

\newcommand{\xra}{\xrightarrow}

\newcommand{\onto}{\twoheadrightarrow}
\newcommand{\into}{\hookrightarrow}

\renewcommand{\geq}{\geqslant}
\renewcommand{\leq}{\leqslant}

\newcommand{\Ext}[4][R]{\operatorname{Ext}_{#1}^{#2}(#3,#4)}

\newcommand{\Otimes}[3][R]{#2\otimes_{#1}#3}
\renewcommand{\Hom}[3][R]{\operatorname{Hom}_{#1}(#2,#3)}	
	
\newcommand{\Tor}[4][R]{\operatorname{Tor}^{#1}_{#2}(#3,#4)}

\newcommand{\width}{\operatorname{width}}

\newcommand{\Comp}[2]{\widehat{#1}^{\ideal{#2}}}

\newcommand{\md}[1]{#1^{\vee}}
\newcommand{\mdp}[1]{(#1)^{\vee}}
\newcommand{\mdd}[1]{#1^{\vee \vee}}

\newcommand{\bidual}[1]{\delta_{#1}}

\newcommand{\tors}[2]{\Gamma_{\mathfrak{#1}}(#2)}

\numberwithin{equation}{lem}

\begin{document}

\bibliographystyle{amsplain}

\author{Bethany Kubik}

\address{Bethany Kubik, Department of Mathematics,
NDSU Dept \#2750,
PO Box 6050,
Fargo, ND 58108-6050
USA}

\email{bethany.kubik@ndsu.edu}

\urladdr{http://www.ndsu.edu/pubweb/\~{}betkubik/}

\author{Micah J. Leamer}

\address{Micah J. Leamer, Department of Mathematics,
University of Nebraska-Lincoln,
PO Box 880130,
Lincoln, NE 68588-0130
USA}

\email{s-mleamer1@math.unl.edu}
\author{Sean Sather-Wagstaff}

\address{Sean Sather-Wagstaff, Department of Mathematics,
NDSU Dept \#2750,
PO Box 6050,
Fargo, ND 58108-6050
USA}

\email{sean.sather-wagstaff@ndsu.edu}

\urladdr{http://www.ndsu.edu/pubweb/\~{}ssatherw/}

\thanks{This material is based on work supported by North Dakota EPSCoR and 
National Science Foundation Grant EPS-0814442.
Micah Leamer was supported by a GAANN grant from the Department of Education.
Sean Sather-Wagstaff was supported  by a grant from the NSA}


\title 
{Homology of artinian and Matlis reflexive modules, I}



\begin{abstract}
Let $R$ be a commutative local noetherian ring, and let $L$ and $L'$ be  $R$-modules.
We investigate the properties of the functors $\Tor{i}{L}{-}$ and $\Ext{i}{L}{-}$.
For instance, we show the following:

\

\begin{enumerate}[\noindent(a)]
\item if $L$ is artinian and $L'$ is noetherian, then $\Hom{L}{L'}$ has finite length;
\item if $L$ and $L'$ are artinian, then $\Otimes{L}{L'}$ has finite length;
\item if $L$ and $L'$ are artinian, then  $\Tor{i}{L}{L'}$ is artinian, and $\Ext{i}{L}{L'}$ is noetherian over the completion $\comp R$;
and
\item if $L$ is artinian and $L'$ is Matlis reflexive, then 
$\Ext{i}{L}{L'}$, $\Ext{i}{L'}{L}$, and $\Tor{i}{L}{L'}$ are Matlis reflexive.
\end{enumerate}

\

\noindent Also, we study the vanishing behavior of these functors,
and we include computations demonstrating the sharpness of our results.
\end{abstract}

\maketitle

\section*{Introduction} 

Throughout this paper, let $R$ be a commutative noetherian local ring
with maximal ideal $\m$ and residue field $k=R/\m$.
The $\m$-adic completion of $R$ is denoted $\comp R$,
the injective hull of $k$ is $E=E_R(k)$, and the Matlis duality functor is
$\mdp{-}=\Hom{-}{E}$.

This paper is concerned, in part, with properties of the functors 
$\Hom{A}{-}$ and $\Otimes{A}{-}$, where $A$ is an artinian $R$-module.
For instance,  the following result is contained in 
Corollaries~\ref{cor100319d} and~\ref{cor100416b}.

\begin{intthm}\label{intthm100928a}
Let $A$, $A'$ and $N$ be $R$-modules such that $A$ and $A'$ are artinian and $N$ is noetherian. 
Then the modules $\Hom{A}{N}$
and 
$\Otimes{A}{A'}$ have finite length.
\end{intthm}

This follows from the fact that $\Otimes{A}{A'}$ can be described as the tensor product of two
finite length modules, 
and an analogous description holds for $\Hom{A}{N}$.

In light  of Theorem~\ref{intthm100928a}, it is natural to investigate the properties of 
$\Ext{i}{A}{-}$ and $\Tor{i}{A}{-}$. 
In general, the modules $\Ext{i}{A}{N}$ and $\Tor{i}{A}{A'}$
will not have finite length. However, we have the following;
see Theorems~\ref{thm100308a} and~\ref{thm100320b}.

\begin{intthm}\label{intthm100928b}
Let $A$ be an artinian $R$-module, and let $i\geq 0$.
Let 
$L$ and $L'$ be  $R$-modules such that 
$\mu_R^i(L)$
and 
$\beta^R_i(L')$
are finite. 
Then $\Ext{i}{A}{L}$ is a noetherian $\comp R$-module,
and $\Tor{i}{A}{L'}$ is  artinian.
\end{intthm}

In this result, we are using the $i$th Bass number  $\mu_R^i(L):=\len_R(\Ext{i}{k}{L})$
and 
the $i$th Betti number  $\beta^R_i(L'):=\len_R(\Tor{i}{k}{L'})$. 
For instance, these are both finite for all $i$ when $L$ and $L'$ are either artinian or noetherian.
In particular, when $A$ and $A'$ are artinian, Theorem~\ref{intthm100928b} implies that
$\Ext{i}{A}{A'}$ is a noetherian $\comp R$-module.
The next result, contained in Theorem~\ref{prop100317a}, gives another explanation for
this fact.

\begin{intthm}\label{intthm100928c}
Let $A$ and $A'$ be artinian $R$-modules, and let $i\geq 0$.
Then there is an isomorphism $\Ext{i}{A}{A'}\cong\Ext[\comp R]{i}{\md{A'}}{\md{A}}$.
Hence, there are noetherian $\comp R$-modules $N$ and $N'$
such that $\Ext{i}{A}{A'}\cong\Ext[\comp R]{i}{N}{N'}$.
\end{intthm}

This result proves useful for studying the vanishing of $\Ext{i}{A}{A'}$,
since the vanishing of $\Ext[\comp R]{i}{N}{N'}$ is somewhat well understood.

Our next result shows how extra conditions on the modules in Theorem~\ref{intthm100928b}
imply that 
$\Ext{i}{A}{L}$ and $\Tor{i}{A}{L'}$ are Matlis reflexive; see Corollaries~\ref{prop100316a}
and~\ref{thm100424a}.

\begin{intthm}\label{intthm100930a}
Let $A$, $L$, and $L'$ be $R$-modules such that $A$ is artinian.
Assume that $R/(\ann_R(A)+\ann_R(L))$ and $R/(\ann_R(A)+\ann_R(L'))$ are complete.
Given  an index $i\geq 0$ such that $\mu^i_R(L)$ and $\beta^R_i(L')$ are finite,
the modules $\Ext{i}{A}{L}$ and $\Tor{i}{A}{L'}$ are Matlis reflexive.
\end{intthm}

A key point in the proof of this theorem is a result of 
Belshoff, Enochs, and Garc{\'{\i}}a Rozas~\cite{belshoff:gmd}:
An $R$-module $M$
is Matlis reflexive if and only if it is mini-max and $R/\ann_R(M)$ is complete.
Here $M$ is \emph{mini-max} when
$M$ has a noetherian submodule $N$ such that $M/N$ is artinian.
In particular, noetherian modules are mini-max, as are artinian modules.

The last result singled out for this introduction
describes the Matlis dual
of $\Ext{i}{M}{M'}$ in some special cases.
It is contained in
Corollary~\ref{cor100618c}. 

\begin{intthm}\label{intthm100930b}
Let $M$ and $M'$ be mini-max $R$-modules, and fix an index $i\geq 0$.  
If either $M$ or $M'$ is Matlis reflexive,
then 
$\md{\Ext{i}{M}{M'}}\cong\Tor{i}{M}{M'^\vee}$.
\end{intthm}

We do not include a description of
the Matlis dual of 
$\Tor{i}{M}{M'}$, as a standard application of Hom-tensor adjointness
shows that
$\md{\Tor{i}{M}{M'}}\cong\Ext{i}{M}{\md{M'}}$.

Many of our results generalize to the non-local setting. As this generalization 
requires additional tools,
we treat it separately in~\cite{kubik:hamm2}.

\section{Background material and preliminary results}\label{sec1}

\subsection*{Torsion Modules}

\begin{defn}\label{defn100207a}
Let $\fa$ be a proper ideal of $R$. We 
denote the $\fa$-adic completion of $R$ by $\Comp Ra$.
Given  an $R$-module $L$,
set
$\Gamma_{\fa}(L)=\{x\in L\mid\text{$\fa^nx=0$ for $n\gg 0$}\}$.
We say that $L$ is \emph{$\fa$-torsion} if $L=\Gamma_{\fa}(L)$.
We set $\supp_R(L)=\{\p\in\spec(R)\mid L_{\p}\neq 0\}$.
\end{defn}

\begin{fact}\label{para0z}
Let $\fa$ be a proper ideal of $R$, and let $L$ be an $\fa$-torsion $R$-module.
\begin{enumerate}[(a)]
\item\label{para0z1}
Every artinian $R$-module is $\m$-torsion.
In particular, the module $E$ is $\m$-torsion.
\item\label{para0z2}
We have $\supp_R(L)\subseteq V(\fa)$.
Hence, if $L$ is $\m$-torsion, then  $\supp_R(L)\subseteq \{\m\}$.
\item\label{para0z3}
The module $L$ has an $\Comp Ra$-module structure that is
compatible with its $R$-module structure, as follows. For each $x\in L$, 
fix an exponent $n$ such that $\fa^nx=0$.
For each 
$r\in\Comp Ra$, the isomorphism $\Comp Ra/\fa^n \Comp Ra\cong R/\fa^n$
provides an element $r_0\in R$ such that $r-r_0\in\fa^n \Comp Ra$, and we set
$rx:=r_0x$.
\item\label{para0z4}
If $R/\fa$ is complete, then $\Comp Ra$ is naturally isomorphic to $\comp R$.
\end{enumerate}
\end{fact}

\begin{lem}\label{lem100423a}
Let $\fa$ be a proper ideal of $R$,
and let $L$ be an $\fa$-torsion $R$-module.
\begin{enumerate}[\rm(a)]
\item \label{lem100423a2}
A subset $Z\subseteq L$ is an $R$-submodule if and only if it is an $\Comp Ra$-submodule.
\item \label{lem100423a3}
The module $L$ is noetherian over $R$ if and only if it is noetherian over $\Comp Ra$.
\end{enumerate}
\end{lem}

\begin{proof}
\eqref{lem100423a2}
Every $\Comp Ra$-submodule of $L$ is an $R$-submodule by restriction of scalars.
Conversely, fix an $R$-submodule
$Z\subseteq L$.
Since $L$ is $\fa$-torsion, so is $Z$, and Fact~\ref{para0z}\eqref{para0z3} implies that $Z$
is an $\Comp Ra$-submodule.

\eqref{lem100423a3}
The set of $R$-submodules of $L$ equals the set of $\Comp Ra$-submodules of $L$,
so they satisfy the ascending chain condition simultaneously.
\end{proof}

\begin{lem}\label{para0'''}
Let $\fa$ be a proper ideal of $R$,
and let $L$ be an $\fa$-torsion $R$-module.
\begin{enumerate}[\rm(a)] 
\item \label{para0e}
The natural map $L\to \Comp Ra\otimes_RL$ is an isomorphism.
\item \label{para0g}
The left and right $\Comp Ra$-module structures on $\Comp Ra\otimes_RL$ 
are the same.
\end{enumerate}
\end{lem}

\begin{proof}
The natural map $L\to \Comp Ra\otimes_RL$ is injective,
as $\Comp Ra$ is faithfully flat over $R$.
To show  surjectivity, it suffices to show that 
each generator
$r\otimes x\in\Otimes{\Comp Ra}{L}$ is of the form $1\otimes x'$ for some $x'\in L$.
Let $n\geq 1$ such that $\fa^nx=0$,
and let $r_0\in R$ such that $r- r_0 \in\fa^n\Comp Ra$.
It follows that
$r\otimes x=r_0\otimes x=1\otimes (r_0x)$,
and this yields the conclusion of part~\eqref{para0e}.
This also proves~\eqref{para0g}
because 
$1\otimes (r_0x)=1\otimes (rx)$.
\end{proof}

\begin{lem} \label{lem100206c1}
Let $\fa$ be a proper ideal of $R$, and let $L$ and $L'$ be  
$R$-modules such that $L$ is $\fa$-torsion.
\begin{enumerate}[\rm(a)]
\item \label{lem100206c1b}
If $L'$ is $\fa$-torsion, then
$\Hom{L}{L'}=\Hom[\Comp{R}a]{L}{L'}$;
thus $\md L=\Hom[\Comp{R}a]{L}{E}$. 
\item \label{lem100206c1a}
One has
$\Hom L{L'}\cong\Hom{L}{\Gamma_{\fa}(L')}=\Hom[\Comp Ra]{L}{\Gamma_{\fa}(L')}$.
\end{enumerate}
\end{lem}

\begin{proof}
\eqref{lem100206c1b}
It suffices to verify the inclusion $\Hom{L}{L'}\subseteq\Hom[\Comp{R}a]{L}{L'}$.
Let $x\in L$ and  $r\in\Comp Ra$,
and fix $\psi\in\Hom{L}{L'}$.
Let $n\geq 1$ such that $\fa^nx=0$ and $\fa^n\psi(x)=0$.
Choose an
element $r_0\in R$ such that $r-r_0\in\fa^n\Comp Ra$.
It follows that
$\psi(rx)
=\psi(r_0x)=r_0\psi(x)=r\psi(x)$;
hence $\psi\in\Hom[\Comp{R}a]{L}{L'}$.

\eqref{lem100206c1a}
For each $f\in\Hom{L}{L'}$,
one has $\im(f)\subseteq\Gamma_{\fa}(L')$. This yields the desired isomorphism,
and the equality is from part~\eqref{lem100206c1b}.
\end{proof}

\subsection*{A Natural Map 
from $\mathbf{Tor_{i}^R(L,L'^{\vee})}$ to $\mathbf{\md{Ext^{i}_R(L,L')}}$}

\begin{defn}\label{defn100602a}
Let $L$  be an $R$-module, and let $J$ be an $R$-complex.
The \emph{Hom-evaluation} morphism
$$\theta_{LJE}\colon\Otimes{L}{\Hom{J}{E}}\to\Hom{\Hom{L}{J}}{E}$$
is given by $\theta_{LJE}(l\otimes\psi)(\phi)=\psi(\phi(l))$.
\end{defn}

\begin{disc}\label{disc100602a}
Let $L$ and $L'$ be $R$-modules, and let $J$ be an injective resolution of $L'$.
Using the notation $\md{(-)}$, we have
$\theta_{LJE}\colon\Otimes{L}{\md{J}}\to\md{\Hom{L}{J}}$.
The complex $\md{J}$ is a flat resolution of $\md{L'}$;
see, e.g., \cite[Theorem~3.2.16]{enochs:rha}.
This explains the first isomorphism in the following sequence:
\begin{align*}
\Tor{i}{L}{\md{L'}}
\xra\cong\HH_i(\Otimes{L}{\md{J}})
\xra{\HH_i(\theta_{LJE})}
&\HH_i(\md{\Hom{L}{J}})
\xra\cong\md{\Ext{i}{L}{L'}}.
\end{align*}
For the second isomorphism, the exactness of $\md{(-)}$ implies that
$\HH_i(\md{\Hom{L}{J}})
\cong\md{\HH^{i}(\Hom{L}{J})}
\cong\md{\Ext{i}{L}{L'}}$.
\end{disc}

\begin{defn}\label{defn100602b}
Let $L$ and $L'$ be $R$-modules, and let $J$ be an injective resolution of $L'$.
The $R$-module homomorphism
$$\Theta^{i}_{LL'}\colon\Tor{i}{L}{\md{L'}}\to\md{\Ext{i}{L}{L'}}$$
is defined to be the composition of the the maps displayed in Remark~\ref{disc100602a}.
\end{defn}

\begin{disc}\label{disc100602b}
Let $L$, $L'$, and $N$ be $R$-modules such that $N$ is noetherian.
It is straightforward to show that the map $\Theta^{i}_{LL'}$ is natural in $L$ and in $L'$.

The fact that $E$ is injective implies that
$\Theta^{i}_{NL'}$ is an isomorphism;
see~\cite[Lemma~3.60]{rotman:iha}.
This explains the first of the following isomorphisms:
\begin{align*}
\md{\Ext{i}{N}{L'}}
&\cong\Tor{i}{N}{\md{L'}}&
\md{\Tor{i}{L}{L'}}
&\cong\Ext{i}{L}{\md{L'}}.
\end{align*}
The second isomorphism is a consequence of Hom-tensor adjointness,
\end{disc}

\subsection*{Numerical Invariants}

\begin{defn}\label{defn100414a}
Let  $L$ be an $R$-module.
For each integer $i$,
the $i$th \emph{Bass number} of $L$ and the $i$th \emph{Betti number} of $L$ are
respectively
\begin{align*}
\mu^i_R(L)&=\len_R(\Ext{i}{k}{L})
&\beta_i^R(L)&=\len_R(\Tor{i}{k}{L})
\end{align*}
where $\len_R(L')$ denotes the length of an $R$-module $L'$.
\end{defn}

\begin{disc}\label{disc100414a}
Let  $L$ be an $R$-module.
\begin{enumerate}[(a)]
\item\label{disc100414a2}
If $I$ is a minimal injective resolution of $L$,
then for each index $i\geq 0$ such that $\mu^i_R(L)<\infty$, we have
$I^{i}\cong E^{\mu^i_R(L)}\oplus J^{i}$ where $J^{i}$ does not have $E$ as a summand,
that is, $\Gamma_{\m}(J^{i})=0$; see, e.g., \cite[Theorem 18.7]{matsumura:crt}.
Similarly, the Betti numbers of a noetherian module are the ranks of the free modules
in a minimal free resolution. 
The situation for Betti numbers of non-noetherian modules is more subtle; see, e.g., Lemma~\ref{lem100206a}.
\item\label{disc100414a3}
Then
$\mu^i_R(L)<\infty$ for all $i\geq 0$ if and only if $\beta_i^R(L)<\infty$
for all $i\geq 0$; see~\cite[Proposition 1.1]{lescot:spmi}.
\end{enumerate}
\end{disc}

When $\fa=\m$, the next invariants can be interpreted in terms of (non)vanishing
Bass and Betti numbers.

\begin{defn}\label{defn100616a}
Let $\fa$ be an ideal of $R$. For each $R$-module $L$, set
\begin{align*}
\depth_R(\fa;L)
&=\inf\{i\geq 0\mid\Ext{i}{R/\fa}{L}\neq 0\}
\\
\width_R(\fa;L)
&=\inf\{i\geq 0\mid\Tor{i}{R/\fa}{L}\neq 0\}.
\end{align*}
We write
$\depth_R(L)=\depth_R(\m;L)$
and $\width_R(L)=\width_R(\m;L)$.
\end{defn}

Part~\eqref{lem100213c5} of the next result is known. We include it for ease of reference.

\begin{lem} \label{lem100213c}
Let $L$ be an $R$-module, and let $\fa$ be an ideal of $R$.
\begin{enumerate}[\rm(a)]
\item\label{lem100213c4}
Then $\width_R(\fa;L)=\depth_R(\fa;\md{L})$ and $\width_R(\fa;\md{L})= \depth_R(\fa;L)$.
\item\label{lem100213c5}
For each index $i\geq 0$ we have
$\beta^R_i(L)=\mu^i_R(\md L)$ and
$\beta^R_i(\md L)=\mu^i_R(L)$.
\item\label{lem100213c1}
$L=\fa L$ if and only if $\depth_R(\fa;\md L)>0$.
\item\label{lem100213c2}
$\md L=\fa (\md L)$ if and only if $\depth_R(\fa;L)>0$.
\item\label{lem100213c3}
$\depth_R(\fa;L)>0$ if and only if $\fa$ contains a non-zero-divisor for $L$.
\end{enumerate}
\end{lem}

\begin{proof}
Part~\eqref{lem100213c4} is from~\cite[Proposition~4.4]{foxby:daafuc},
and part~\eqref{lem100213c5} follows directly from this.

\eqref{lem100213c1}--\eqref{lem100213c2}
These follow from part~\eqref{lem100213c4} since $L=\fa L$ if and only if $\width_R(\fa;L)>0$.

\eqref{lem100213c3}
By definition, we need to show that $\Hom{R/\fa}{L}=0$
if and only if $\fa$ contains a non-zero-divisor for $L$.
One implication is explicitly stated in~\cite[Proposition~1.2.3(a)]{bruns:cmr}.
One can prove the converse  like~\cite[Proposition~1.2.3(b)]{bruns:cmr},
using the fact that $R/\fa$ is finitely generated. 
\end{proof}

The next result  characterizes artinian modules in terms of Bass numbers.

\begin{lem}\label{lem100423c}
Let $L$ be an $R$-module. The following conditions are equivalent:
\begin{enumerate}[\rm(i)]
\item \label{lem100423c1}
$L$ is an artinian $R$-module;
\item \label{lem100423c2}
$L$ is an artinian $\comp R$-module;
\item \label{lem100423c3}
$\comp{R}\otimes_R L$ is an artinian $\comp{R}$-module; and
\item \label{lem100423c0}
$L$ is $\m$-torsion and $\mu^0_R(L)<\infty$.
\end{enumerate}
\end{lem}

\begin{proof}
\eqref{lem100423c1}$\iff$\eqref{lem100423c0}
If $M$ is artinian over $R$, then it is $\m$-torsion by Fact~\ref{para0z}\eqref{para0z1},
and we have $\mu^0_R(L)<\infty$ by~\cite[Theorem~3.4.3]{enochs:rha}.
For the converse, assume that $L$ is $\m$-torsion and $\mu^0=\mu^0_R(L)<\infty$.
Since $L$ is $\m$-torsion, so is $E_R(L)$.
Thus, we have $E_R(L)\cong E^{\mu^0}$, which is artinian since $\mu^0<\infty$.
Since $L$ is a submodule of the artinian module $E_R(L)$, it is also artinian.

To show the equivalence of the conditions~\eqref{lem100423c1}--\eqref{lem100423c3},
first note that each of these conditions implies that $L$ is $\m$-torsion.
(For condition~\eqref{lem100423c3}, use the monomorphism $L\to\comp R\otimes_RL$.)
Thus, for the rest of the proof, we assume that $L$ is $\m$-torsion.

Because of the equivalence \eqref{lem100423c1}$\iff$\eqref{lem100423c0},
it suffices to show that
$$\mu^0_R(L)=\mu^0_{\comp R}(L)=\mu^0_{\comp R}(\Otimes{\comp R}{L}).$$
These equalities follow from the next isomorphisms
\begin{align*}
\Hom{k}{L}
&\cong\Hom[\comp R]{k}{L}
\cong\Hom[\comp R]{k}{\Otimes{\comp R}L}
\end{align*}
which are from 
Lemmas~\ref{lem100206c1}\eqref{lem100206c1b}
and~\ref{para0'''}, respectively.
\end{proof}

\begin{lem}\label{lem100423f}
Let $L$ be an $R$-module.
\begin{enumerate}[\rm(a)]
\item \label{lem100423f1}
The module $L$ is noetherian over $R$ if and only if $\md L$ is artinian over $R$.
\item \label{lem100423f3}
If $\md L$ is noetherian over $R$ or over $\comp R$, then
$L$ is artinian over $R$.
\item \label{lem100423f2}
Let $\fa$ be a proper ideal of $R$ such that 
$R/\fa$ is complete.
If $L$ is $\fa$-torsion, then
$L$ is artinian over $R$ if and only if $\md L$ is noetherian over $R$.
\end{enumerate}
\end{lem}

\begin{proof}
\eqref{lem100423f1}
Assume first that $L$ is noetherian. Then $L$ is a homomorphic image of $R^b$ for some integer $b\geq 0$.
It follows that $\md L$ is isomorphic to a submodule of the artinian module
$\md{(R^{b})}\cong E^b$, so $\md L$ is artinian.

For the converse, assume that $\md L$ is artinian, and fix an ascending chain
$L_1\subseteq L_2\subseteq\cdots$ of submodules of $L$.
Dualize the surjections
$L\onto L/L_1\onto L/L_2\onto\cdots$ 
to obtain a sequence of monomorphisms 
$\cdots\into\md{(L/L_2)}\into\md{(L/L_1)}\into\md{L}$.
The corresponding descending chain of submodules must stabilize since $\md L$ is artinian, 
and it follows that the original chain 
$L_1\subseteq L_2\subseteq\cdots$ of submodules of $L$
also stabilizes. Thus $L$ is noetherian.

\eqref{lem100423f3}
Argue as in part~\eqref{lem100423f1}.

\eqref{lem100423f2}
Assume that $L$ is $\fa$-torsion.
One implication is from part~\eqref{lem100423f3}. For the converse, assume that $L$ is artinian over $R$.
Lemma~\ref{lem100423c} shows that
$\md L=\Hom[\comp R]{L}{E}$ is artinian over  $\comp R$;
see Lemma~\ref{lem100206c1}\eqref{lem100206c1b}.
From~\cite[Theorem~18.6(v)]{matsumura:crt} we know that $L$ is noetherian over $\comp R$, so
Lemma~\ref{lem100423a}\eqref{lem100423a3} implies that
$L$ is noetherian over $R$.
\end{proof}

\subsection*{Mini-max and Matlis Reflexive Modules}

\begin{defn} \label{defn100423a}
An $R$-module $M$ is \emph{mini-max}
if there is a noetherian submodule $N\subseteq M$ such that $M/N$ is artinian.
\end{defn}

\begin{defn} \label{notn100204a}
An $R$-module $M$ is \emph{Matlis reflexive}
provided that the natural biduality map $\bidual{M}\colon M\to\mdd{M}$,
given by $\bidual{M}(x)(\psi)=\psi(x)$, is an isomorphism.
\end{defn}

\begin{fact}\label{para1}
An $R$-module $M$ is Matlis reflexive
if and only if it is mini-max and $R/\ann_R(M)$ is complete;
see~\cite[Theorem~12]{belshoff:gmd}.
Thus, if $M$ is mini-max over $R$, then $\Otimes{\comp R}{M}$ is
Matlis reflexive over $\comp R$.
\end{fact}

\begin{lem} \label{lem100206a}
If $M$ is  mini-max over $R$,
then $\beta^R_i(M),\mu^i_R(M)<\infty$ for all $i\geq 0$.
\end{lem}

\begin{proof}
We show that $\mu^i_R(M)<\infty$ for all $i\geq 0$; then
Remark~\ref{disc100414a}\eqref{disc100414a3} 
implies that $\beta_i^R(M)<\infty$ for all $i\geq 0$.
The noetherian case is standard. 
If $M$ is artinian, then we have $\mu^0=\mu^0_R(M)<\infty$ by
Lemma~\ref{lem100423c};
since $E^{\mu^0}$ is artinian, an induction argument shows that
$\mu^i_R(M)<\infty$ for all $i\geq 0$.
One deduces the mini-max case from the artinian and noetherian cases,
using a long exact sequence.
\end{proof}

\begin{lem}\label{lem100423d}
Let $L$ be an $R$-module such that $R/\ann_R(L)$ is complete.
The following conditions are equivalent:
\begin{enumerate}[\rm(i)]
\item \label{lem100423d1}
$L$ is Matlis reflexive over $R$;
\item \label{lem100423d3}
$L$ is mini-max over $R$; 
\item \label{lem100423d4}
$L$ is mini-max over $\comp R$; and
\item \label{lem100423d2}
$L$ is Matlis reflexive over $\comp R$.
\end{enumerate}
\end{lem}

\begin{proof}
The equivalences \eqref{lem100423d1}$\iff$\eqref{lem100423d3}
and \eqref{lem100423d4}$\iff$\eqref{lem100423d2} are from Fact~\ref{para1}.
Note that conditions~\eqref{lem100423d4}
and~\eqref{lem100423d2} make sense since $L$ is an $\comp R$-module;
see Fact~\ref{para0z}.

\eqref{lem100423d3}$\implies$\eqref{lem100423d4}
Assume that $L$ is mini-max over $R$, and fix a noetherian $R$-sub-module
$N\subseteq L$
such that  $L/N$ is artinian over $R$. 
As $R/\ann_R(N)$ is complete,
Fact~\ref{para0z}\eqref{para0z4} and Lemma~\ref{lem100423a}\eqref{lem100423a2}
imply that $N$ is an $\comp R$-submodule.
Similarly,
Lemmas~\ref{lem100423a}\eqref{lem100423a3}
and~\ref{lem100423c}
imply that $N$ is noetherian over $\comp R$,
and $L/N$ is an artinian over $\comp R$.
Thus $L$ is mini-max over~$\comp R$.

\eqref{lem100423d4}$\implies$\eqref{lem100423d3}
Assume that $L$ is mini-max over $\comp R$, and fix a
noetherian $\comp R$-submodule $L'\subseteq L$
such that $L/L'$ is artinian over $\comp R$.
Lemmas~\ref{lem100423a}\eqref{lem100423a3} and~\ref{lem100423c}
imply that $L'$ is noetherian over $R$,
and $L/L'$ is artinian over $R$, so $L$ is mini-max over $R$.
\end{proof}

\begin{lem}\label{lem100213b}
Let $L$ be an $R$-module such that $\fm^tL=0$ for some integer $t\geq 1$.
Then the following conditions are equivalent:
\begin{enumerate}[\rm(i)]
\item \label{lem100213b0}
$L$ is mini-max over $R$ (equivalently, over $\comp R$);
\item \label{lem100213b1}
$L$ is artinian over $R$ (equivalently, over $\comp R$);
\item \label{lem100213b2}
$L$ is noetherian over $R$ (equivalently, over $\comp R$); and
\item \label{lem100213b3}
$L$ has finite length over $R$ (equivalently, over $\comp R$).
\end{enumerate}
\end{lem}

\begin{proof}
Lemma~\ref{lem100423d} shows that 
$L$ is mini-max over $R$ if and only if it is mini-max over $\comp R$.
Also, $L$ is artinian (resp., noetherian or finite length) over $R$ if and only if it 
is artinian (resp., noetherian or finite length) over $\comp R$ by Lemmas~\ref{lem100423c} 
and~\ref{lem100423a}\eqref{lem100423a3}.

The equivalence of conditions~\eqref{lem100213b1}--\eqref{lem100213b3}
follows from an application of~\cite[Proposition 2.3.20]{enochs:rha} over
the artinian ring $R/\m^t$. The implication \eqref{lem100213b1}$\implies$\eqref{lem100213b0}
is evident. For the implication
\eqref{lem100213b0}$\implies$\eqref{lem100213b1},
assume that $L$ is mini-max over $R$.
Given a noetherian submodule $N\subseteq L$ such that $L/N$ is artinian,
the implication \eqref{lem100213b2}$\implies$\eqref{lem100213b1}
shows that $N$ is artinian; hence so is $L$.
\end{proof}

\begin{lem} \label{lem100617a}
The class of mini-max (resp., noetherian, artinian, finite length, or Matlis reflexive) $R$-modules
is closed under submodules, quotients, and extensions.
\end{lem}

\begin{proof}
The noetherian, artinian, and finite length cases are standard,
as is the Matlis reflexive case; see~\cite[p. 92, Exercise 2]{enochs:rha}.
For the mini-max case,
fix an exact sequence $0\to L'\xra{f} L\xra{g} L''\to 0$.
Identify $L'$ with $\im(f)$.
Assume first that $L$ is mini-max, and fix a noetherian submodule $N$ such that $L/N$ 
is artinian.  Then $L'\cap N$ is noetherian, and the quotient
$L'/(L'\cap N)\cong (L'+N)/N$ is artinian, since it is a submodule of $L/N$. 
Thus $L'$ is mini-max.  
Also, $(N+L')/L'$ is noetherian and
$[L/L']/[(N+L')/L']\cong L/(N+L')$ is artinian,
so $L''\cong L/L'$ is mini-max.

Next, assume that $L'$ and $L''$ are mini-max,
and fix noetherian submodules $N'\subseteq L'$ and $N''\subseteq L''$ 
such that $L'/N'$ and 
$L''/N''$ are artinian. Let $x_1,\hdots ,x_h$ be coset representatives in $L$ of a generating set 
for $N''$. 
Let $N=N'+Rx_1+\hdots +Rx_h$.  Then $N$ is noetherian and the following
commutative diagram has exact rows:
\[
\xymatrix{ 0\ar[r]  & N\cap L'\ar @{^{(}->} [d]  \ar[r] & N \ar[r] \ar @{^{(}->} [d] & N''\ar @{^{(}->} [d] \ar[r] & 0\\
0\ar[r] &L'\ar[r] &L\ar[r] &L'' \ar[r] & 0.}
\]
The sequence
$0\to L'/(N\cap L')\to L/N\to L''/N''\to 0$ is exact
by the Snake Lemma. 
The module $L'/(N\cap L')$ is artinian, being a quotient of $L'/N'$.
Since the class of  artinian modules is closed
under extensions, the module $L/N$ is 
artinian.  It follows that $L$ is mini-max.
\end{proof}

The next two lemmas apply to the classes of modules from Lemma~\ref{lem100617a}.

\begin{lem} \label{lem100430a}
Let $\catc$ be a class  $R$-modules
that is closed under submodules, quotients, and extensions.
\begin{enumerate}[\rm(a)]
\item  \label{lem100430a2}
Given an exact sequence $L'\xra{f} L\xra{g} L''$,
if $L',L''\in\catc$, then $L\in\catc$.
\item  \label{lem100430a3}
Given an $R$-complex $X$ 
and  an integer $i$, if $X_i\in\catc$,
then $\HH_i(X)\in\catc$.
\item  \label{lem100430a4}
Given a noetherian $R$-module $N$, if $L\in\catc$, then 
$\Ext{i}{N}{L},\Tor iNL\in\catc$.
\end{enumerate}
\end{lem}

\begin{proof}
\eqref{lem100430a2}
Assume that $L',L''\in \catc$. By assumption,
$\im(f),\im(g)\in \catc$.
Using the exact sequence
$0\to\im(f)\to L\to\im(g)\to 0$,
we conclude that $L$ is in $\catc$.

\eqref{lem100430a3}
The module $\HH_i(X)$ is a subquotient of $X_i$, so it is in $\catc$
by assumption.

\eqref{lem100430a4}
If $F$ is a minimal free resolution of $N$, then the modules in the complexes
$\Hom FL$ and $\Otimes FL$ are in $\catc$, so their homologies are in $\catc$ by
part~\eqref{lem100430a3}.
\end{proof}

\begin{lem} \label{lem100614a}
Let $R\to S$ be a local ring homomorphism, and 
let $\catc$ be a class of $S$-modules
that is closed under submodules, quotients, and extensions.
Fix an $S$-module $L$, an $R$-module $L'$, an $R$-submodule $L''\subseteq L'$, and an index $i\geq 0$.
\begin{enumerate}[\rm(a)]
\item  \label{lem100614a1}
If $\Ext{i}{L}{L''},\Ext{i}{L}{L'/L''}\in\catc$, then $\Ext{i}{L}{L'}\in\catc$.
\item  \label{lem100614a2}
If $\Ext{i}{L''}{L},\Ext{i}{L'/L''}{L}\in\catc$, then $\Ext{i}{L'}{L}\in\catc$.
\item  \label{lem100614a3}
If $\Tor{i}{L}{L''}, \Tor{i}{L}{L'/L''}\in\catc$, then $\Tor{i}{L}{L'}\in\catc$.
\end{enumerate}
\end{lem}

\begin{proof}
We prove part~\eqref{lem100614a1}; the other parts are proved similarly.
Apply $\Ext{i}{L}{-}$ to the exact sequence 
$0\to L''\to L'\to L'/L''\to 0$
to obtain the next exact sequence:
\begin{equation*}
\Ext{i}L{L''}\to\Ext{i}L{L'}\to\Ext{i}L{L'/L''}.
\end{equation*}
Since $L$ is  an $S$-module,  the maps in this
sequence are $S$-module homomorphisms.
Now, apply Lemma~\ref{lem100430a}\eqref{lem100430a2}.
\end{proof}

\section{Properties of $\Ext{i}{M}{-}$}\label{sec2}
This section documents properties of the functors $\Ext{i}{M}{-}$ where $M$ is a mini-max $R$-module. 

\subsection*{Noetherianness of $\mathbf{Ext^{i}_R(A,L)}$}

\begin{lem} \label{lem100206c2}
Let  $A$ and $L$ be  $R$-modules such that $A$ is artinian and $L$ is $\m$-torsion.
\begin{enumerate}[\rm(a)]
\item \label{lem100206c2a}
Then
$\Hom{L}{A}=\Hom[\comp R]{L}{A}\cong\Hom[\comp R]{\md{A}}{\md L}$.
\item \label{lem100206c2b}
If $L$ is artinian, then $\Hom{L}{A}$ 
is a noetherian $\comp{R}$-module.
\end{enumerate}
\end{lem}

\begin{proof}
\eqref{lem100206c2a}
The first equality is from
Lemma~\ref{lem100206c1}\eqref{lem100206c1b}.
For the second equality,
the fact that $A$ is Matlis reflexive over $\comp R$ explains the first step below:
\begin{align*}
\Hom[\comp R]{L}{A}
&
\cong\Hom[\comp R]{L}{A^{vv}}
\cong\Hom[\comp R]{A^{v}}{L^{v}}
\cong\Hom[\comp R]{\md A}{\md L}
\end{align*}
where $(-)^v=\Hom[\comp R]{-}{E}$.
The second step follows from Hom-tensor adjointness,
and the third step is from Lemma~\ref{lem100206c1}\eqref{lem100206c1b}.

\eqref{lem100206c2b}
If $L$ is artinian, then $\md L$ and $\md A$ are noetherian over $\comp R$,
so $\Hom[\comp R]{\md{A}}{\md{L}}$ is also noetherian over $\comp R$.
\end{proof}

The next result contains part of Theorem~\ref{intthm100928b} from the introduction.
When $R$ is not complete, the 
example $\Hom EE\cong \comp R$ shows that $\Ext i{A}{L}$ is not necessarily
noetherian  or artinian over $R$.

\begin{thm}\label{thm100308a}
Let $A$ and $L$ be $R$-modules
such that $A$ is artinian. For each 
index $i\geq 0$ such that $\mu^i_R(L)<\infty$, the module
$\Ext{i}{A}{L}$
is a noetherian $\comp{R}$-module.
\end{thm}

\begin{proof}
Let $J$ be
a minimal $R$-injective resolution of $L$.
Remark~\ref{disc100414a}\eqref{disc100414a2} implies that $\Gamma_{\m}(J)^{i}\cong E^{\mu^i_R(L)}$.
Lemma~\ref{lem100206c1}\eqref{lem100206c1a}
explains the first isomorphism below:
$$\Hom{A}{J}^{i}\cong\Hom{A}{\Gamma_{\m}(J)^{i}}
\cong\Hom{A}{E}^{\mu^i_R(L)}.$$
Lemma~\ref{lem100206c2}
implies that
these are noetherian $\comp{R}$-modules.
The differentials in the complex $\Hom A{\tors{m}J}$ are $\comp R$-linear
because $A$ is an $\comp R$-module.
Thus, the subquotient $\Ext{i}{A}{L}$ is a noetherian $\comp{R}$-module.
\end{proof}

\begin{cor}\label{cor100501a}
Let $A$ and $M$ be $R$-modules such that $A$ is artinian 
and $M$ is mini-max. For each index
$i\geq 0$, the module
$\Ext{i}{A}{M}$
is a noetherian $\comp{R}$-module.
\end{cor}

\begin{proof}
Apply Theorem~\ref{thm100308a} and Lemma~\ref{lem100206a}.
\end{proof}

The next result contains part of Theorem~\ref{intthm100930a} from the introduction.

\begin{cor} \label{prop100316a}
Let $A$ and $L$ be  $R$-modules
such that $R/(\ann_R(A)+\ann_R(L))$ is  complete and $A$ is artinian. 
For each 
index $i\geq 0$ such that $\mu^i_R(L)<\infty$, the module $\Ext{i}{A}{L}$ is noetherian
and Matlis reflexive over $R$ and $\comp R$.
\end{cor}

\begin{proof}
\hyphenation{re-flex-ive}
Theorem~\ref{thm100308a}   shows 
that $\Ext{i}{A}{L}$ is noetherian over $\comp R$; so, it is Matlis reflexive over $\comp R$.
As $\ann_R(A)+\ann_R(L)\subseteq\ann_R(\Ext{i}{A}{L})$, 
Lemmas~\ref{lem100423a}\eqref{lem100423a3} and~\ref{lem100423d} imply
that $\Ext{i}{A}{L}$ is noetherian and Matlis reflexive over $R$.
\end{proof}

\begin{cor}\label{prop100315d}
Let $A$ and $L$ be  $R$-modules such that 
$R/(\ann_R(A)+\ann_R(L))$ is artinian
and $A$ is artinian. 
Given an index $i\geq 0$ such that $\mu^i_R(L)<\infty$, one has 
$\len_R(\Ext iA{L})<\infty$.
\end{cor}

\begin{proof}
Apply Theorem~\ref{thm100308a} and
Lemma~\ref{lem100213b}.
\end{proof}

\subsection*{Matlis Reflexivity of $\mathbf{Ext^{i}_R(M,M')}$}

\begin{thm} \label{prop100420a}
Let $A$ and $M$ be $R$-modules such that $A$ is artinian and $M$ is mini-max.
For each $i\geq 0$,
the module $\Ext{i}M{A}$ is  Matlis reflexive over $\comp R$.
\end{thm}

\begin{proof}
Fix a noetherian submodule $N\subseteq M$ such that $M/N$ is artinian.
Since $A$ is artinian, it is an $\comp R$-module.
Corollary~\ref{cor100501a} implies that $\Ext{i}{M/N}{A}$ is a noetherian $\comp R$-module.
As $\Ext{i}N{A}$ is artinian,
Lemma~\ref{lem100614a}\eqref{lem100614a2} says that
$\Ext{i}M{A}$ is a mini-max $\comp R$-module
and hence is Matlis reflexive over $\comp R$ by Fact~\ref{para1}.
\end{proof}

\begin{thm} \label{prop100420b}
Let $M$ and $N'$ be   $R$-modules such that $M$ is mini-max and $N'$ is noetherian. 
Fix an index $i\geq 0$.
If $R/(\ann_R(M)+\ann_R(N'))$ is complete, then 
$\Ext{i}M{N'}$ is  noetherian and Matlis reflexive over $R$ and over $\comp R$.
\end{thm}

\begin{proof}
Fix a noetherian submodule $N\subseteq M$ such that $M/N$ is artinian.
If the ring $R/(\ann_R(M)+\ann_R(N'))$ is complete, then so is 
$R/(\ann_R(M/N)+\ann_R(N'))$.
Corollary~\ref{prop100316a} implies that $\Ext{i}{M/N}{N'}$ is noetherian over $R$.
Since $\Ext{i}N{N'}$ is noetherian over $R$, Lemma~\ref{lem100614a}\eqref{lem100614a2} implies that
$\Ext{i}M{N'}$ is  noetherian over $R$.
As $R/(\ann_R(\Ext{i}M{N'}))$ is complete, 
Fact~\ref{para1} implies that $\Ext{i}M{N'}$ is  
also Matlis reflexive over $R$.
Thus
$\Ext{i}M{N'}$ is  noetherian and Matlis reflexive over $\comp R$
by Lemmas~\ref{lem100423a}\eqref{lem100423a3} and~\ref{lem100423d}.
\end{proof}

\begin{thm} \label{prop100421b}
Let $M$ and $M'$ be mini-max $R$-modules, and fix an index $i\geq 0$.
\begin{enumerate}[\rm(a)]
\item \label{prop100421b2}
If $R/(\ann_R(M)+\ann_R(M'))$ is complete, then 
$\Ext{i}M{M'}$ is  Matlis reflexive over $R$ and $\comp R$.
\item \label{prop100421b3}
If $R/(\ann_R(M)+\ann_R(M'))$ is artinian, then 
$\Ext{i}M{M'}$ has finite length.
\end{enumerate}
\end{thm}

\begin{proof}
Fix a noetherian submodule $N'\subseteq M'$ such that $M'/N'$ is artinian.

\eqref{prop100421b2}
Assume that $R/(\ann_R(M)+\ann_R(M'))$ is complete. 
Theorem~\ref{prop100420b} implies that
the module  $\Ext{i}M{N'}$ is Matlis reflexive 
over $R$.
Theorem~\ref{prop100420a} shows that
$\Ext{i}M{M'/N'}$ is Matlis reflexive 
over $\comp R$; hence, it is  Matlis reflexive 
over $R$ by Lemma~\ref{lem100423d}.
Thus, Lemmas~\ref{lem100614a}\eqref{lem100614a1} and~\ref{lem100423d} imply that
$\Ext{i}M{M'}$ is  Matlis reflexive  over $R$ and $\comp R$.

\eqref{prop100421b3}
This follows from part~\eqref{prop100421b2},
because of Fact~\ref{para1} and Lemma~\ref{lem100213b}.
\end{proof}

A special case of the next result  can be found in \cite[Theorem 3]{belshoff:scrtmrm}.

\begin{cor} \label{cor100618}
Let $M$ and $M'$ be $R$-modules such that $M$ is mini-max and $M'$ is Matlis reflexive. For each index $i\geq 0$,
the modules
$\Ext{i}M{M'}$ and $\Ext{i}{M'}M$ are Matlis reflexive over $R$ and $\comp R$.
\end{cor}

\begin{proof}
Apply Theorem~\ref{prop100421b}\eqref{prop100421b2}
and Fact~\ref{para1}.
\end{proof}

\subsection*{Length Bounds for $\mathbf{Hom_R(A,L)}$}

%
%

\begin{lem}\label{lem100312b}
Let $A$ and $L$ be $R$-modules
such that $A$ is artinian and  $\fm^n\Gamma_{\m}(L)=0$ for some $n\geq 1$.
Fix an index $t\geq 0$ such that $\m^tA=\m^{t+1}A$, and let $s$ be an integer such that 
$s\geq\min(n,t)$.
Then  
\begin{align*}
\Hom{A}{L}
&
\cong\Hom{A/\m^sA}{L}
\cong\Hom{A/\m^sA}{(0:_L\m^s)}.
\end{align*}
\end{lem}

\begin{proof}
Given any map $\psi\in\Hom{A/\m^sA}{L}$, the image of $\psi$
is annihilated by $\m^s$. That is,
$\im(\psi)\subseteq(0:_L\m^s)$; hence 
$\Hom{A/\m^sA}{L}
\cong\Hom{A/\m^sA}{(0:_L\m^s)}$.
In the next sequence, the first and third isomorphisms are from Lemma~\ref{lem100206c1}\eqref{lem100206c1a}:
\begin{align*}
\Hom{A}{L}
&\cong\Hom{A}{\tors mL}
\cong\Hom{A/\m^sA}{\tors mL}
\cong\Hom{A/\m^sA}{L}.
\end{align*}
For the second isomorphism, we argue by cases.
If $s\geq n$, then we have $\m^s\Gamma_{\m}(L)=0$ because
$\m^n\Gamma_{\m}(L)=0$, and the isomorphism is evident.
If $s<n$, then we have $n>s\geq t$, so $\m^tA=\m^sA=\m^{n}A$ since
$\m^tA=\m^{t+1}A$; it follows that
$\Hom{A}{\tors mL}
\cong\Hom{A/\m^nA}{\tors mL}
\cong\Hom{A/\m^sA}{\tors mL}$.
\end{proof}

For the next result,
the example $\Hom EE\cong\comp R$ shows that the condition 
$\fm^n\Gamma_{\m}(L)=0$ is necessary.

\begin{thm}\label{thm100312b}
Let $A$ and $L$ be $R$-modules 
such that $A$ is artinian and $\fm^n\Gamma_{\m}(L)=0$ for some $n\geq 1$.
Fix an index $t\geq 0$ such that $\m^tA=\m^{t+1}A$, and let $s$ be an integer such that 
$s\geq\min(n,t)$.
Then there is an inequality
$$\len_R(\Hom{A}{L})\leq\beta^R_0(A)\len_R(0:_L\m^s).$$
Here, we use the convention $0\cdot\infty=0$.
\end{thm}

\begin{proof}
We deal with the degenerate case first. 
If $\beta^R_0(A) =0$, then $A/\m A=0$,
so 
$$\Hom AL\cong\Hom{A/\m A}{L}=\Hom 0L=0$$
by Lemma~\ref{lem100312b}.
So, we assume for the rest of the proof that
$\beta^R_0(A)\neq 0$.
We also assume without loss of generality that
$\len_R(0:_L\m^s)<\infty$.

Lemma~\ref{lem100312b}
explains the first step in the following sequence:
\begin{align*}
\len_R(\Hom{A}{L})
&=\len_R(\Hom{A/\m^sA}{(0:_L\m^s)})\\
&\leq\beta^R_0(A/\m^s A)\len_R(0:_L\m^s)\\
&=\beta^R_0(A)\len_R(0:_L\m^s).
\end{align*}
The second step  can be proved by induction on 
$\beta^R_0(A/\m^s A)$ and $\len_R(0:_L\m^s)$.
\end{proof}

The next result gives part of Theorem~\ref{intthm100928a} from the introduction.
Example~\ref{ex14} shows that 
one should not expect to have $\len_R(\Ext{i}{A}{N})<\infty$ when $i\geq 1$.

\begin{cor}\label{cor100319d}
If $A$ and $N$ are  $R$-modules such that $A$ is artinian and $N$ is noetherian,
then $\len_R(\Hom{A}{N})
<\infty$.
\end{cor}

\begin{proof}
Apply Theorem~\ref{thm100312b} and Lemma~\ref{lem100206a}.
\end{proof}

\section{Properties of $\Tor{i}{M}{-}$}\label{sec3}

This section focuses on properties of the functors $\Tor{i}{M}{-}$ where $M$ is a mini-max $R$-module.

\subsection*{Artinianness of $\mathbf{Tor_{i}^R(A,L)}$}

\

\vspace{1mm}

\noindent 
The next result contains part of Theorem~\ref{intthm100928b} from the introduction.
Recall that a module is artinian over $R$
if and only if it is artinian over $\comp R$; see Lemma~\ref{lem100423c}.

\begin{thm}\label{thm100320b}
Let $A$ and $L$ be $R$-modules such that $A$ is artinian. For each 
index $i\geq 0$ such that $\beta^R_i(L)<\infty$, the module
$\Tor{i}{A}{L}$
is artinian.
\end{thm}

\begin{proof}
Lemma~\ref{lem100213c}\eqref{lem100213c5} implies
that $\mu^i_R(\md L)=\beta^R_i(L)<\infty$.
By Remark~\ref{disc100602b}, we have
$\Ext{i}{A}{\md L}\cong\md{\Tor{i}{A}{L}}$.
Thus, $\md{\Tor{i}{A}{L}}$ is a noetherian $\comp R$-module by 
Theorem~\ref{thm100308a}, and we conclude that $\Tor{i}{A}{L}$ is artinian
by Lemma~\ref{lem100423f}\eqref{lem100423f3}.
\end{proof}

For the next result, 
the example $\Otimes{E}{R}\cong E$ shows that
$\Tor{i}{A}{L}$ is not necessarily noetherian over $R$ or $\comp R$.

\begin{cor}\label{cor100502a}
Let $A$ and $M$ be  $R$-modules
such that $A$ is artinian and $M$   mini-max. For 
each index $i\geq 0$, the module
$\Tor{i}{A}{M}$
is artinian.
\end{cor}

\begin{proof}
Apply Theorem~\ref{thm100320b} and Lemma~\ref{lem100206a}.
\end{proof}

The proofs of the next two results are similar to those of Corollaries~\ref{prop100316a}
and~\ref{prop100315d}.
The first result contains part of Theorem~\ref{intthm100930a} from the introduction.

\begin{cor}\label{thm100424a}
Let $A$ and $L$ be $R$-modules such that 
$R/(\ann_R(A)+\ann_R(L))$ is complete and $A$ is artinian. 
For each 
index $i\geq 0$ such that $\beta^R_i(L)<\infty$, the module
$\Tor{i}{A}{L}$
is artinian and Matlis reflexive over $R$ and  $\comp{R}$.
\end{cor}

\begin{cor}\label{prop100416a}
Let $A$ and $L$ be  $R$-modules such that 
$R/(\ann_R(A)+\ann_R(L))$ is artinian and $A$ is artinian.
Given an index $i\geq 0$ such that $\beta_i^R(L)<\infty$, one has 
$\len_R(\Tor iA{L})<\infty$.
\end{cor}

\subsection*{$\mathbf{Tor^R_{i}(M,M')}$ is Mini-max}

\begin{thm} \label{prop100424b}
Let $M$ and $M'$ be mini-max $R$-modules, and fix an index $i\geq 0$.
\begin{enumerate}[\rm(a)]
\item \label{prop100424b1}
The $R$-module
$\Tor{i}M{M'}$ is mini-max over $R$.
\item \label{prop100424b2}
If $R/(\ann_R(M)+\ann_R(M'))$ is complete, then 
$\Tor{i}M{M'}$ is  Matlis reflexive over $R$ and $\comp R$.
\item \label{prop100424b3}
If $R/(\ann_R(M)+\ann_R(M'))$ is artinian, then 
$\Tor{i}M{M'}$ has finite length.
\end{enumerate}
\end{thm}

\begin{proof}
\eqref{prop100424b1}
Choose a noetherian submodule $N\subseteq M$ such that $M/N$ is artinian.
Lemmas~\ref{lem100617a} and~\ref{lem100430a}\eqref{lem100430a4} say
that $\Tor{i}N{M'}$ is mini-max.
Corollary~\ref{cor100502a} implies that $\Tor{i}{M/N}{M'}$ mini-max, so
$\Tor{i}M{M'}$ is mini-max by Lemma~\ref{lem100614a}\eqref{lem100614a3}.

Parts~\eqref{prop100424b2}
and~\eqref{prop100424b3} now follow from Lemmas~\ref{lem100423d}
and~\ref{lem100213b}.
\end{proof}

A special case of the next result is contained in~\cite[Theorem 3]{belshoff:scrtmrm}.

\begin{cor} \label{cor100618b}
Let $M$ and $M'$ be  $R$-modules such that $M$ is mini-max and $M'$ is Matlis reflexive. For each  index $i\geq 0$,
the module
$\Tor{i}M{M'}$ is  Matlis reflexive over $R$ and $\comp R$.
\end{cor}

\begin{proof}
Apply Theorem~\ref{prop100424b}\eqref{prop100424b2}
and Fact~\ref{para1}.
\end{proof}

\subsection*{Length Bounds for $\mathbf{A\otimes_RL}$}

%
%

\begin{lem} \label{lemma1}
Let $A$ be an artinian module, and let $\fa$ be a proper ideal of $R$.
Fix an integer $t\geq 0$ such that $\mathfrak{a}^tA=\mathfrak{a}^{t+1}A$. 
Given an $\fa$-torsion $R$-module $L$,
one has 
$$A\otimes_RL\cong (A/\mathfrak{a}^tA)\otimes_R L
\cong\Otimes{(A/\mathfrak{a}^tA)}{(L/\mathfrak{a}^tL)}.$$
\end{lem}

\begin{proof}
The isomorphism $(A/\mathfrak{a}^tA)\otimes_R L
\cong\Otimes{(A/\mathfrak{a}^tA)}{(L/\mathfrak{a}^tL)}$
is from the following:
\begin{align*}
(A/\mathfrak{a}^tA)\otimes_R L
&\cong\Otimes{[\Otimes{(A/\mathfrak{a}^tA)}{(R/\mathfrak{a}^t)}]}{L}\\
&\cong\Otimes{(A/\mathfrak{a}^tA)}{[\Otimes{(R/\mathfrak{a}^t)}{L}]}\\
&\cong\Otimes{(A/\mathfrak{a}^tA)}{(L/\mathfrak{a}^tL)}.
\end{align*}
For the isomorphism $A\otimes_RL\cong (A/\mathfrak{a}^tA)\otimes_R L$,
consider the exact sequence:
$$0\to\fa^tA\to A\to A/\fa^tA\to 0.$$
The exact sequence induced by $-\otimes_RL$ has the form
\begin{equation}\label{lemma1a}
(\fa^tA)\otimes_RL\to A\otimes_RL\to (A/\fa^tA)\otimes_RL\to 0.
\end{equation}
The fact that $L$ is $\fa$-torsion and  $\fa^tA=\fa^{t+i}A$ for all $i\geq 1$ implies that
$\Otimes{(\fa^t A)}{L}=0$, so the sequence~\eqref{lemma1a} yields the desired isomorphism.
\end{proof}

The example $\Otimes{E}{R}\cong R$ shows that the $\m$-torsion assumption 
on $L$ is necessary in the next result.

\begin{thm} \label{cor28}
Let $A$  be an artinian $R$-module, and let $L$ be an $\m$-torsion $R$-module.
Fix an integer $t\geq 0$ such that $\m^tA=\m^{t+1}A$.
Then there are inequalities
\begin{align}
\label{cor28a}
\len_R(A\otimes_RL)
&\leq \len_R\left(A/\m^tA\right)\beta^R_0(L)\\
\label{cor28b}
\len_R(A\otimes_RL) 
&\leq \beta^R_0(A)\len_R\left(L/\m^tL\right).
\end{align}
Here we use the convention $0\cdot\infty=0$.
\end{thm}

\begin{proof}
From Lemma~\ref{lemma1} we have
\begin{equation}\label{cor28c}
A\otimes_RL
\cong\Otimes{(A/\mathfrak{m}^tA)}{(L/\mathfrak{m}^tL)}.
\end{equation}
Lemmas~\ref{lem100206a} and~\ref{lem100213b}
imply that $\len_R(A/\m^tA)<\infty$
and $\beta^R_0(A)<\infty$.

For the degenerate cases, first note that $\len_R(A/\m^tA)=0$ if and only if $\beta^R_0(A)=0$.
When $\len_R(A/\m^tA)=0$, the  isomorphism~\eqref{cor28c}
implies that $A\otimes_RL=0$; hence the desired inequalities.
Thus, we assume without loss of generality that 
$1\leq\beta^R_0(A)\leq\len_R(A/\m^tA)$.
Further, we assume that $\beta^R_0(L)<\infty$.

The isomorphism~\eqref{cor28c} provides the first step in the next sequence:
$$\len_R(A\otimes_RL)=\len_R(\Otimes{(A/\mathfrak{m}^tA)}{(L/\mathfrak{m}^tL)})
\leq\len_R(A/\mathfrak{m}^tA)\beta^R_0(L).$$
The second step in this sequence can be verified by induction on 
$\len_R(A/\mathfrak{m}^tA)$ and $\beta^R_0(L)$.
This explains the inequality~\eqref{cor28a}, and~\eqref{cor28b} is verified similarly.
\end{proof}

The next result contains part of Theorem \ref{intthm100928a} from the introduction.
Example~\ref{ex100420a} shows that 
one should not expect to have $\len_R(\Tor{i}{A}{A'})<\infty$ when $i\geq 1$.

\begin{cor} \label{cor100416b}
If $A$ and $A'$ are  artinian $R$-modules,
then 
$\len_R(A\otimes_RA')<\infty$.
\end{cor}

\begin{proof}
Apply Theorem~\ref{cor28} and Lemmas~\ref{lem100206a} and~\ref{lem100213b}.
\end{proof}

\section{The Matlis dual of $\Ext{i}{L}{L'}$} \label{sec100601a}

This section contains the proof of Theorem~\ref{intthm100930b} from the introduction;
see Corollary~\ref{cor100618c}.
Most of the section is devoted to  technical results for use in the proof.

\begin{lem} \label{lem100317a}
Let $L$ be an $R$-module.
If $I$ is an $R$-injective resolution of $L$,
and $J$ is an $\comp R$-injective resolution of $\Otimes{\comp R}{L}$,
then there is a homotopy equivalence $\Gamma_{\m}(I)\xra\sim\Gamma_{\m}(J)=\Gamma_{\m\comp R}(J)$.
\end{lem}

\begin{proof}
Each injective $\comp R$-module $J'$ is injective over $R$;
this follows from the  isomorphism
$
\Hom{-}{J'}\cong\Hom{-}{\Hom[\comp R]{\comp R}{J'}}\cong\Hom[\comp R]{\Otimes{\comp R}{-}}{J'}
$
since $\comp R$ is flat over $R$.
Hence, there is a lift $f\colon I\to J$ of the natural map $\xi\colon L\to \Otimes{\comp R}L$.
This lift is a chain map of $R$-complexes. 

We show that the induced map 
$\Gamma_{\m}(f)\colon \Gamma_{\m}(I)\to \Gamma_{\m}(J)=\Gamma_{\m\comp R}(J)$
is a homotopy equivalence.
As  $\Gamma_{\m}(I)$ and $\Gamma_{\m}(J)$
are bounded above complexes of injective $R$-modules, it suffices to show that
$\Gamma_{\m}(f)$ induces an isomorphism
on  homology in each degree.
The induced map on homology is compatible with the following sequence:
$$\HH^i(\Gamma_{\m}(I))\cong\HH^{i}_{\m}(L)\xra[\cong]{\HH^{i}_{\m}(\xi)}
\HH^{i}_{\m}(\Otimes{\comp R}L)\cong\HH^i(\Gamma_{\m}(J)).
$$
The map 
$\HH^{i}_{\m}(\xi)\colon\HH^{i}_{\m}(L)\to \HH^{i}_{\m}(\Otimes{\comp R}L)$
is an isomorphism
(see the proof of~\cite[Proposition~3.5.4(d)]{bruns:cmr}) so we have the desired 
homotopy equivalence.
\end{proof}

\begin{lem}\label{lem100317b}
Let $L$ and $L'$ be $R$-modules
such  that $L$ is $\m$-torsion.
Then for each index $i\geq 0$, there are $\comp R$-module isomorphisms
$$\Ext{i}{L}{L'}
\cong\Ext{i}{L}{\Otimes{\comp R}{L'}}
\cong\Ext[\comp{R}]{i}{L}{\Otimes{\comp R}{L'}}.$$
\end{lem}

\begin{proof}
Let $I$ be
an $R$-injective resolution of $L'$, and
let $J$ be
an $\comp R$-injective resolution of $\Otimes{\comp R}{L'}$.
Because $L$ is $\m$-torsion,
Lemma~\ref{lem100206c1}\eqref{lem100206c1a}
explains the first, third and sixth steps in the next display:
\begin{align*}
\Hom{L}{I}
&\cong\Hom{L}{\Gamma_{\m}(I)}
\sim\Hom{L}{\Gamma_{\m}(J)}
\cong\Hom{L}{J}\\
 \Hom{L}{\Gamma_{\m}(J)}
&=\Hom{L}{\Gamma_{\m\comp R}(J)}
=\Hom[\comp R]{L}{\Gamma_{\m\comp R}(J)}
\cong\Hom[\comp R]{L}{J}.
\end{align*}
The homotopy equivalence in the second step is from Lemma~\ref{lem100317a}.
The fifth step is from Lemma~\ref{lem100206c1}\eqref{lem100206c1b}.
Since $L$ is $\m$-torsion, it is an $\comp R$-module, so the isomorphisms
and the homotopy equivalence in this sequence are $\comp R$-linear.
In particular, the complexes $\Hom{L}{I}$ and $\Hom{L}{J}$ and $\Hom[\comp R]{L}{J}$
have isomorphic cohomology over $\comp R$, so one has the desired isomorphisms.
\end{proof}

The next result contains Theorem~\ref{intthm100928c} from the introduction.
It shows, for instance, that given artinian $R$-modules $A$ and $A'$,
there are noetherian $\comp R$-modules $N$ and $N'$ such that
$\Ext{i}{A}{A'}\cong\Ext[\comp R]{i}{N}{N'}$;
thus, it provides an alternate proof of Corollary~\ref{cor100501a}.

\begin{thm}\label{prop100317a}
Let $A$ and $M$ be $R$-modules
such  that $A$ is artinian and $M$ is mini-max. Then for each index $i\geq 0$, we have
$\Ext{i}{A}{M}\cong\Ext[\comp{R}]{i}{\md M}{\md A}.$
\end{thm}

\begin{proof}
Case 1: $R$ is complete. Let $F$ be a free resolution of $A$. 
It follows that each $F_i$ is flat, so the complex
$\md F$ is an injective resolution of $\md A$; see~\cite[Theorem~3.2.9]{enochs:rha}.
We obtain the isomorphism $\Ext{i}{A}{M}
\cong\Ext{i}{\md M}{\md A}$ by taking cohomology
in the next sequence:
\begin{align*}
\Hom{F}{M}
&\cong\Hom{F}{\mdd{M}}\cong\Hom{\md{M}}{\md{F}}.
\end{align*}
The first step follows from the fact that
$M$ is Matlis reflexive; see Fact~\ref{para1}. 
The second step is from
Hom-tensor adjointness

Case 2: the general case.
The first step below is  from Lemma~\ref{lem100317b}:
$$\Ext{i}{A}{M}\cong\Ext[\comp{R}]{i}{A}{\Otimes{\comp R}{M}}
\cong\Ext[\comp{R}]{i}{(\Otimes{\comp R}{M})^{v}}{A^{v}}
\cong\Ext[\comp{R}]{i}{\md M}{\md A}.$$
Here $(-)^{v}=\Hom[\comp R]{-}{E}$.  
Since $M$ is mini-max,
it follows that $\Otimes{\comp R}{M}$ is mini-max over $\comp R$. 
Thus, the second step is from Case 1.
For the third step use Hom-tensor adjointness 
and Lemma~\ref{lem100206c1}\eqref{lem100206c1b} to see that $(\Otimes{\comp R}{M})^{v}\cong M^{\vee}$
and $A^{v}\cong A^{\vee}$.
\end{proof}

\begin{fact}\label{fact100604a}
Let $L$ and $L'$ be $R$-modules, and fix an index $i\geq 0$. Then the following diagram commutes,
where the unlabeled isomorphism is from Remark~\ref{disc100602b}:
$$\xymatrix@C=1.5cm{
\Ext{i}{L'}{L}
\ar[r]^-{\bidual{\Ext{i}{L'}{L}}}
\ar[d]_{\Ext{i}{L'}{\bidual L}}
& \mdd{\Ext{i}{L'}{L}}
\ar[d]^{\md{(\Theta_{L'L}^{i})}}
\\
\Ext{i}{L'}{\mdd L}
\ar[r]^-{\cong}
&\md{\Tor{i}{L'}{\md L}}.
}$$
\end{fact}

\begin{lem}\label{lem100604a}
Let $L$ be an $R$-module, and fix an index $i\geq 0$.
If $\mu^i_R(L)<\infty$, then the map
$\Ext{i}{k}{\bidual L}\colon \Ext{i}{k}{L}\to\Ext{i}{k}{\mdd L}$
is an isomorphism.
\end{lem}

\begin{proof}
The assumption $\mu^i_R(L)<\infty$ says that
$\Ext{i}{k}{L}$ is a finite dimensional $k$-vector space, 
so it is Matlis reflexive over $R$; that is, the map
$$\bidual{\Ext{i}{k}{L}}\colon \Ext{i}{k}{L}\to \mdd{\Ext{i}{k}{L}}$$
is an isomorphism.
Since $k$ is finitely generated, Remark~\ref{disc100602b} implies that
$$\Theta^{i}_{kL} \colon\Tor{i}{k}{\md L}\to\md{\Ext{i}{k}{L}}$$
is an isomorphism. Hence $\md{(\Theta_{kL}^i)}$ is also an isomorphism.
Using  Fact~\ref{fact100604a} with $L'=k$,
we conclude that $\Ext{i}{k}{\bidual L}$ is an isomorphism, as desired.
\end{proof}

\begin{lem}\label{lem100604b}
Let $A$ and $L$ be $R$-modules such that $A$ is artinian.  
Fix an index $i\geq 0$ such that 
$\mu_R^{i-1}(L)$, $\mu_R^i(L)$ and $\mu_R^{i+1}(L)$ are finite.
Then the map
$$\Ext{i}{A}{\bidual L}\colon \Ext{i}{A}{L}\to\Ext{i}{A}{\mdd L}$$
is an isomorphism.
\end{lem}

\begin{proof}
Lemma~\ref{lem100604a} implies that for $t=i-1,i,i+1$ the maps
$$\Ext{t}{k}{\bidual L}\colon \Ext{t}{k}{L}\to\Ext{t}{k}{\mdd L}$$
are isomorphisms.
As the biduality map $\bidual L$ is injective,  we have
an exact sequence
\begin{equation} 
\label{lem100604b1}
0\to L\xra{\bidual L}\mdd{L}\to\coker{\bidual L}\to 0.
\end{equation}
Using the long exact sequence associated to $\Ext{}{k}{-}$,
we conclude that for $t=i-1,i$ we have
$\Ext{t}{k}{\coker{\bidual L}}=0$.
In other words, we have $\mu^{t}_R(\coker{\bidual L})=0$.

Let $J$ be a minimal injective resolution of $\coker{\bidual L}$. The previous paragraph
shows that for $t=i-1,i$ the module $J^{t}$ does not have $E$ as a summand by Remark~\ref{disc100414a}\eqref{disc100414a2}. That is,
we have $\Gamma_{\m}(J^{t})=0$, so
Lemma~\ref{lem100206c1}\eqref{lem100206c1a} implies that
$$\Hom{A}{J^t}\cong\Hom{A}{\Gamma_{\m}(J^t)}=0.$$
It follows that
$\Ext{t}{A}{\coker(\bidual L)}=0$ for $t=i-1,i$. 
From the long exact sequence associated to $\Ext{}{A}{-}$
with respect to~\eqref{lem100604b1}, it follows that $\Ext{i}{A}{\bidual L}$
is an isomorphism, as desired.
\end{proof}

We are now ready to tackle the main results of this section.

\begin{thm}\label{prop100601b}
Let $A$ and $L$ be $R$-modules such that $A$ is artinian.  
Fix an index $i\geq 0$ such that 
$\mu_R^{i-1}(L)$, $\mu_R^i(L)$ and $\mu_R^{i+1}(L)$ are finite.
\begin{enumerate}[\rm(a)]
\item\label{prop100601b2}
There is an $R$-module isomorphism 
$\Ext{i}{A}{L}^v\cong\Tor{i}{A}{\md L}$
where $(-)^v=\Hom[\comp R]{-}{E}$.
\item\label{prop100601b1}
If $R/(\ann_R(A)+\ann_R(L))$ is complete, then $\Theta^{i}_{AL}$
provides an isomorphism $\Tor{i}{A}{\md{L}}\cong\md{\Ext{i}{A}{L}}$.
\end{enumerate}
\end{thm}

\begin{proof}
\eqref{prop100601b1}  
Corollary~\ref{prop100316a} and Lemma~\ref{lem100604b} show that the  maps
\begin{align*}
\bidual{\Ext{i}{A}{L}}
&\colon\Ext{i}{A}{L}\to\mdd{\Ext{i}{A}{L}}
\\
\Ext{i}{A}{\bidual L}
&\colon \Ext{i}{A}{L}\to\Ext{i}{A}{\mdd L}
\end{align*}
are isomorphisms.
Fact~\ref{fact100604a} implies that
$\md{(\Theta^{i}_{AL})}$ is an isomorphism,
so we conclude that $\Theta^{i}_{AL}$ is also an isomorphism.

\eqref{prop100601b2}  
Lemma~\ref{lem100317b} explains the first step in the next sequence:
\begin{align*}
\Ext{i}{A}{L}^v
&\cong
\Ext[\comp R]{i}{A}{\Otimes{\comp R}{L}}^v\\
&\cong
\Tor[\comp R]{i}{A}{(\Otimes{\comp R}{L})^v}\\
&\cong
\Tor{i}{A}{(\Otimes{\comp R}{L})^v}\\
&\cong
\Tor{i}{A}{\md L}.
\end{align*}
The second step is  from part~\eqref{prop100601b1}, 
as $\comp R$ is complete and $\mu_{\comp R}^t(\Otimes{\comp R}{L})=\mu_R^t(L)<\infty$
for $t=i-1,i,i+1$. 
The fourth step is from
Hom-tensor adjointness.
For the third step,
let $P$ be a projective resolution of $A$ over $R$.
Since $\comp R$ is flat over $R$, the complex $\Otimes{\comp R}{P}$ is a 
projective resolution of $\Otimes{\comp R}{A}\cong A$ over $\comp R$;
see Lemma~\ref{para0'''}\eqref{para0e}. Thus, the third step follows from the isomorphism
$\Otimes[\comp R]{(\Otimes{\comp R}{P})}{(\Otimes{\comp R}{L})^v}\cong\Otimes{P}{(\Otimes{\comp R}{L})^v}$.
\end{proof}

\begin{question}
Do the conclusions of Lemma~\ref{lem100604b} and Theorem~\ref{prop100601b} 
hold when one only assumes
that $\mu^i_R(L)$ is finite?
\end{question}

\begin{cor}\label{cor100602a}
Let $A$ and $M$ be $R$-modules such that $A$ is artinian and $M$ is mini-max.  
For each index $i\geq 0$, one has
$\Ext{i}{A}{M}^v\cong\Tor{i}{A}{\md M}$,
where $(-)^v=\Hom[\comp R]{-}{E}$.
\end{cor}

\begin{proof}
Apply Theorem~\ref{prop100601b}\eqref{prop100601b2} and Lemma~\ref{lem100206a}.
\end{proof}

\begin{thm}\label{prop100601a}
Let $M$ and $M'$ be mini-max $R$-modules, and fix an index $i\geq 0$.  
If $R/(\ann_R{(M)}+\ann_R{(M')})$ is complete,
then 
$\Theta^{i}_{MM'}$
is an isomorphism, so 
$$\Ext{i}{M}{M'}^v=\md{\Ext{i}{M}{M'}}\cong\Tor{i}{M}{M'^\vee}$$
where $(-)^v=\Hom[\comp R]{-}{E}$.
\end{thm}

\begin{proof}
Theorem~\ref{prop100421b}\eqref{prop100421b2} implies that
$\Ext{i}{M}{M'}$ is Matlis reflexive over $R$, so Lemma~\ref{lem100206c1}\eqref{lem100206c1b}
and Fact~\ref{para1} 
imply that $\Ext{i}{M}{M'}^v=\md{\Ext{i}{M}{M'}}$. Thus, it remains to show that
$\Theta^{i}_{MM'}$ is an isomorphism.

Case 1: $M$ is noetherian.  In the next sequence, the first and last steps are from Hom-tensor adjointness.
The second step is standard since $M$ is noetherian:
\begin{align*}
\Ext{i}{M}{M'}^\vee
&\cong
(\Otimes{\comp R}{\Ext{i}{M}{M'}})^v\\
&\cong
\Ext[\comp R]{i}{\Otimes{\comp R}{M}}{\Otimes{\comp R}{M'}}^v\\
&\cong
\Tor[\comp R]{i}{\Otimes{\comp R}{M}}{(\Otimes{\comp R}{M'})^v}\\
&\cong
\Tor{i}{M}{(\Otimes{\comp R}{M'})^v}\\
&\cong
\Tor{i}{M}{M'^\vee}.
\end{align*}
Since $M$ and $M'$ are mini-max over $R$, 
the modules $\Otimes{\comp R}{M}$ and $\Otimes {\comp R}{M'}$ 
are Matlis reflexive over $\comp R$;
see Fact~\ref{para1}.
Thus~\cite[Theorem 4(c)]{belshoff:mrm} explains the 
third step.  The fourth step is from the fact that $\comp R$ is flat over $R$.
Since these isomorphisms are compatible with $\Theta^{i}_{MM'}$, it follows that
$\Theta^{i}_{MM'}$ is an isomorphism.

Case 2: the general case. Since $M$ is mini-max over $R$,
there is an exact sequence of $R$-modules homomorphisms $0\to N\to M\to A\to 0$ 
such that $N$ is noetherian and $A$ is artinian.  The long 
exact sequences associated to $\Tor{}{-}{\md{M'}}$ and $\Ext{}{-}{M'}$ 
fit into the following commutative diagram:
$$\xymatrix{
\cdots\ar[r]
&\Tor{i}{N}{M'^\vee}\ar[r]\ar[d]^{\Theta^{i}_{NM'}}
&\Tor{i}{M}{M'^\vee}\ar[r]\ar[d]^{\Theta^{i}_{MM'}}
&\Tor{i}{A}{M'^\vee}\ar[d]^{\Theta^{i}_{AM'}}\ar[r]
&\cdots\\
\cdots\ar[r]
&\Ext{i}{N}{M'}^\vee\ar[r]
&\Ext{i}{M}{M'}^\vee\ar[r]
&\Ext{i}{A}{M'}^\vee\ar[r]
&\cdots.
}$$
Case 1 shows that $\Theta^{i}_{NM'}$ and $\Theta^{i-1}_{NM'}$ are  isomorphisms.
Theorem~\ref{prop100601b}\eqref{prop100601b1} implies that 
$\Theta^{i}_{AM'}$ and $\Theta^{i+1}_{AM'}$ are isomorphisms.  
Hence, the Five Lemma shows that $\Theta^{i}_{MM'}$ is an isomorphism.
\end{proof}

The next result contains Theorem~\ref{intthm100930b} from the introduction.
A special case of it  can be found in \cite[Theorem 3]{belshoff:scrtmrm}.

\begin{cor}\label{cor100618c}
Let $M$ and $M'$ be mini-max $R$-modules, and fix an index $i\geq 0$.  
If either $M$ or $M'$ is Matlis reflexive,
then 
$\Theta^{i}_{MM'}$
is an isomorphism, so one has
$\Ext{i}{M}{M'}^v=\md{\Ext{i}{M}{M'}}\cong\Tor{i}{M}{M'^\vee}$,
where $(-)^v=\Hom[\comp R]{-}{E}$.
\end{cor}

\begin{proof}
Apply Theorem~\ref{prop100601a}
and Fact~\ref{para1}.
\end{proof}

The next example shows that the modules $\md{\Ext{i}{L}{L'}}$ 
and $\Tor{i}{L}{\md{L'}}$ are not isomorphic in general.

\begin{ex}\label{ex100601a}
Assume that $R$ is not complete.  We have $\ann_R(E)=0$, 
so the ring $R/\ann_R(E) \cong R$ is not complete, by assumption. 
Thus, Fact~\ref{para1} implies that $E$ is not Matlis reflexive,
that is, the biduality map
$\delta_E\colon E\into\mdd{E}$ is not an isomorphism.
Since $\mdd E$ is injective, we have $\mdd{E}\cong E\oplus J$
for some non-zero injective $R$-module $J$. The uniqueness of 
direct sum decompositions of injective $R$-modules implies that
$\mdd E\not\cong E$. This 
provides the second step below:
$$\md{\Hom{E}{E}}\cong\mdd{E}\not\cong E\cong\Otimes{E}{\comp R}\cong\Otimes{E}{\md{E}}.$$
The third step is from Lemma~\ref{para0'''}\eqref{para0e},
and the remaining steps are standard.
\end{ex}

\section{Vanishing of Ext and Tor} \label{sec100317b}

In this section we describe the sets of associated primes of $\Hom{A}{M}$ and attached primes of
$\Otimes AM$ over $\comp R$. The section concludes with 
some results on the related
topic of vanishing for $\Ext iAM$ and $\Tor iAM$.

\subsection*{Associated and Attached Primes}

\

\vspace{1mm}

\noindent
The following is dual to the notion of associated primes of  noetherian modules;
see, e.g., \cite{macdonald:srmcr} or \cite[Appendix to \S 6]{matsumura:crt} or~\cite{ooishi:mdwm}.

\begin{defn} \label{defn100215a}
Let $A$ be an artinian $R$-module. 
A prime ideal $\p\in\spec(R)$ is \emph{attached} to $A$ if there is a submodule
$A'\subseteq A$ such that $\p=\ann_R(A/A')$. We let $\att_R(A)$ denote the set
of prime ideals attached to $A$.
\end{defn}

\begin{lem} \label{lem100215a}
Let $A$ be an artinian $R$-module such that
$R/\ann_R(A)$ is   complete, and let $N$ be a noetherian $R$-module.
There are equalities
\begin{align*}
\supp_R(\md A)&=\cup_{\p\in\ass_R(\md A)}V(\p)=\cup_{\p\in\att_R(A)}V(\p)
\\
\att_R(\md N)&=\ass_R(N)
\\
\att_R(A)&=\ass_R(\md A). 
\end{align*}
\end{lem}

\begin{proof}
The $R$-module $\md A$ is noetherian by Lemma~\ref{lem100423f}\eqref{lem100423f2}, 
so the first equality is standard,
and the second equality follows from the fourth one.
The third equality is from~\cite[(2.3) Theorem]{sharp:srvlcm}.
This also explains the second step in the next sequence
$$\att_R(A)=\att_R(\mdd A)=\ass_R(\md A)$$
since
$\md A$ is noetherian.
The first step in this sequence follows from the fact that $A$ is Matlis reflexive;
see Fact~\ref{para1}.
\end{proof}

The next proposition can also be deduced from a result of
Melkersson and Schenzel~\cite[Proposition~5.2]{melkersson:cam}.

\begin{prop} \label{cor100319a}
Let $A$ and $L$ be $R$-modules such that $\mu^0_R(L)<\infty$ and $A$ is artinian. Then \label{thm100308b}
\begin{gather*}
\ass_{\comp R}(\Hom A{L})
=\ass_{\comp{R}}(\md A)\cap\supp_{\comp{R}}(\md{\Gamma_{\m}(L)})
=\att_{\comp{R}}(A)\cap\supp_{\comp{R}}(\md{\Gamma_{\m}(L)}).
\end{gather*}
\end{prop}

\begin{proof}
The assumption $\mu^0_R(L)<\infty$ implies that
$\Gamma_{\m}(L)$
is artinian. 
This implies that $\md{\Gamma_{\m}(L)}$ is a noetherian $\comp R$-module,
so a result of Bourbaki~\cite[IV 1.4 Proposition 10]{bourbaki:ac3-4}
provides the third equality in the next sequence; see also~\cite[Exercise~1.2.27]{bruns:cmr}:
\begin{align*}
\ass_{\comp R}(\Hom{A}L)
&=\ass_{\comp R}(\Hom A{\Gamma_{\m}(L)})\\
&=\ass_{\comp R}(\Hom[\comp R]{\md{\Gamma_{\m}(L)}}{\md A})\\
&=\ass_{\comp{R}}(\md A)\cap\supp_{\comp{R}}(\md{\Gamma_{\m}(L)})\\
&=\att_{\comp{R}}(A)\cap\supp_{\comp{R}}(\md{\Gamma_{\m}(L)}).
\end{align*}
The remaining equalities are from Lemmas~\ref{lem100206c1}\eqref{lem100206c1a},
\ref{lem100206c2}\eqref{lem100206c2a}, and~\ref{lem100215a}, respectively.
\end{proof}

\begin{cor} \label{cor100607b}
Let $M$ and $M'$ be mini-max $R$-modules
such that the quotient $R/(\ann_R(M)+\ann_R(M'))$ is complete.
\begin{enumerate}[\rm(a)]
\item \label{cor100607b1}
For each index $i\geq 0$, one has $\Ext{i}{M}{M'}
\cong\Ext{i}{\md{M'}}{\md{M}}$.
\item \label{cor100607b2}
If $M'$ is noetherian, then 
$$\ass_{\comp R}(\Hom{M}{M'})
=\att_{\comp R}(\md{M'})\cap\supp_{\comp R}(\md{\Gamma_{\m}(\md{M})}).$$
\end{enumerate}
\end{cor}

\begin{proof}
\eqref{cor100607b1}
The first step in the next sequence comes from
Theorem~\ref{prop100421b}\eqref{prop100421b2}:
\begin{align*}
\Ext{i}{M}{M'}
&\cong\mdd{\Ext{i}{M}{M'}}
\cong\md{(\Tor{i}{M}{\md{M'}})}
\cong\Ext{i}{\md{M'}}{\md{M}}.
\end{align*}
The remaining steps are from Theorem~\ref{prop100601a}
and Remark~\ref{disc100602b}, respectively. 

\eqref{cor100607b2}
This follows from the case $i=0$ in part~\eqref{cor100607b1}
because of Proposition~\ref{cor100319a}.
\end{proof}

\begin{prop}\label{prop100416d}
Let $A$ and $L$ be $R$-modules
such that $A$ is  artinian and $\beta^R_0(L)<\infty$.
Then
$$\att_{\comp R}(\Otimes{A}{L})
=\ass_{\comp R}(\md A)\cap\supp_{\comp R}(\md{\Gamma_{\m}(\md L)})
=\att_{\comp R}(A)\cap\supp_{\comp R}(\md{\Gamma_{\m}(\md L)}).$$
\end{prop}

\begin{proof}
Theorem~\ref{thm100320b} implies that $\Otimes{A}{L}$ is artinian.
Hence, we have
$$
\Hom[\comp R]{\Otimes{A}{L}}{E}\cong\Hom{\Otimes{A}{L}}{E}
\cong\Hom{A}{\md L}$$
by Lemma~\ref{lem100206c1}\eqref{lem100206c1b}, 
and this explains the second step in the next sequence:
\begin{align*}
\att_{\comp R}(\Otimes{A}{L})
&=\ass_{\comp R}(\Hom[\comp R]{\Otimes{A}{L}}{E})
=\ass_{\comp R}(\Hom{A}{\md L})
\end{align*}
The first step  is from Lemma~\ref{lem100215a}.
Since $\mu^0_R(\md L)<\infty$ by Lemma~\ref{lem100213c}\eqref{lem100213c5},
we obtain the desired equalities from Proposition~\ref{cor100319a}.
\end{proof}

Next, we give an alternate description of the module $\md{\Gamma_{\m}(L)}$
from the previous results.
See Lemma~\ref{lem100215a} for a description of its support.

\begin{disc} \label{disc100319c}
Let $L$ be an $R$-module.
There is an isomorphism
$\md{\Gamma_{\m}(L)}\cong\comp{\md{L}}$.
In particular, given a noetherian $R$-module $N$, 
one has
$\md{\Gamma_{\m}(\md{N})}\cong\Otimes{\comp R}{N}$.
When $R$ is Cohen-Macaulay
with a dualizing module $D$,
Grothendieck's local duality theorem 
implies that
$\md{\Gamma_{\m}(N)}\cong\Otimes{\comp{R}}{\Ext{\dim(R)}{N}{D}}$;
see, e.g., \cite[Theorem~3.5.8]{bruns:cmr}.
A similar description is available when $R$ is not Cohen-Macaulay,
provided that it has a dualizing complex; see~\cite[Chapter V, \S 6]{hartshorne:rad}.
\end{disc}

\subsection*{Vanishing of Hom and Tensor Product}

\

\vspace{1mm}

\noindent
For the next result 
note that if $L$ is noetherian, then the conditions on $\mu^0_R(L)$
and $R/(\ann_R(A)+\ann_R(\Gamma_{\m}(L)))$  are automatically satisfied.
Also, the example $\Hom{E}{E}\cong R$ when $R$ is complete shows the necessity of the
condition on $R/(\ann_R(A)+\ann_R(\Gamma_{\m}(L)))$.

\begin{prop} \label{prop100319b}
Let $A$ be an artinian $R$-module. Let $L$ be an $R$-module
such  that $R/(\ann_R(A)+\ann_R(\Gamma_{\m}(L)))$ is artinian and $\mu^0_R(L)<\infty$.
Then $\Hom AL=0$ if and only if $A=\m A$ or $\Gamma_{\m}(L)= 0$.
\end{prop}

\begin{proof}
If $\Gamma_{\m}(L)=0$, then we are done
by Lemma~\ref{lem100206c1}\eqref{lem100206c1a}, so assume that
$\Gamma_{\m}(L)\neq 0$. 
Theorem~\ref{thm100308a} and Lemma~\ref{lem100213b} show
that $\Hom AL$ has finite length.
Thus Proposition~\ref{cor100319a}
implies that $\Hom AL\neq 0$ if and only if $\m\comp R\in\ass_{\comp R}(\md A)$,
that is, if and only if $\depth_{\comp R}(\md A)=0$.
Lemma~\ref{lem100213c}\eqref{lem100213c1} shows that 
$\depth_{\comp R}(\md A)=0$ if and only if $\m\comp RA\neq A$,
that is, if and only if $\m A\neq A$.
\end{proof}

For the next result 
note that the conditions on $L$ are satisfied when $L$ is artinian.

\begin{prop}\label{prop100419a}
Let $A$ be an artinian $R$-module,
and let $L$ be an $\m$-torsion $R$-module.
The following conditions are equivalent:
\begin{enumerate}[\rm(i)]
\item \label{prop100419a1a} $A\otimes_RL=0$;
\item \label{prop100419a1b} 
either $A=\m A$ or $L=\m L$; and
\item \label{prop100419a1c} 
either $\depth_R(A^{\vee})>0$ or $\depth_R(L^{\vee})>0$.
\end{enumerate}
\end{prop}

\begin{proof}
\eqref{prop100419a1a}$\iff$\eqref{prop100419a1b}
If $\Otimes{A}{L}=0$, then we have
$$0=\len_R(\Otimes{A}{L})\geq\beta^R_0(A)\beta^R_0(L)$$
so either $\beta^R_0(A)=0$ or $\beta^R_0(L)=0$, that is
$A/\m A=0$ or $L/\m L=0$.
Conversely, if $A/\m A=0$ or $L/\m L=0$,
then we have either $\beta^R_0(A)=0$ or $\beta^R_0(L)=0$, so
Theorem~\ref{cor28} implies that 
$\len_R(\Otimes{A}{L})=0$.

The implication 
\eqref{prop100419a1b}$\iff$\eqref{prop100419a1c}
is from Lemma~\ref{lem100213c}\eqref{lem100213c1}.
\end{proof}

The next result becomes simpler when $L$ is artinian, as $\Gamma_{\m}(L)=L$ in this case.

\begin{thm}\label{prop100320a}  
Let $A$ and $L$ be $R$-modules such that $A$ is artinian and 
$\mu^0_R(L)<\infty$.
The following conditions are 
\label{prop100308b}
equivalent:
\begin{enumerate}[\rm(i)]
\item \label{prop100320a3'} 
$\Hom{A}{L}=0$;
\item \label{prop100320a3}  \label{prop100308b3} 
$\Hom{A}{\Gamma_{\m}(L)}=0$;
\item \label{prop100320a7}  \label{prop100308b7} 
$\Hom[\comp R]{\md{\Gamma_{\m}(L)}}{\md{A}}=0$;
\item \label{prop100320a6x}  \label{prop100308b6x} 
there is an element $x\in\ann_{\comp R}(\Gamma_{\m}(L))$ such that $A=xA$; 
\item \label{prop100320a6}  \label{prop100308b6} 
$\ann_{\comp R}(\Gamma_{\m}(L)) A=A$; 
\item \label{prop100320a8}  \label{prop100308b8} 
$\ann_{\comp R}(\Gamma_{\m}(L))$ contains a non-zero-divisor for $\md A$; and
\item \label{prop100320a5}  \label{prop100308b5} 
$\att_{\comp R}(A)\cap\supp_{\comp R}(\md{\Gamma_{\m}(L)})=\emptyset$.
\end{enumerate}
\end{thm}

\begin{proof}
The equivalence \eqref{prop100320a3'}$\iff$\eqref{prop100320a3}
is from Lemma~\ref{lem100206c1}\eqref{lem100206c1a}.
The equivalence \eqref{prop100308b3}$\iff$\eqref{prop100308b5}
follows from Proposition~\ref{thm100308b}, and 
the 
equivalence 
\eqref{prop100308b3}$\iff$\eqref{prop100308b7} follows from 
Lemma~\ref{lem100206c2}\eqref{lem100206c2a}.
The equivalence \eqref{prop100308b6x}$\iff$\eqref{prop100308b8} 
follows from the fact that the map $A\xra{x}A$ is surjective
if and only if the map $\md A\xra{x}\md A$ is injective. The equivalence 
\eqref{prop100308b6}$\iff$\eqref{prop100308b8} follows
from Lemma~\ref{lem100213c}, parts~\eqref{lem100213c1} and~\eqref{lem100213c3}.

The module $\Gamma_{\m}(L)$ is artinian as 
$\mu^0_R(L)<\infty$.
Since $\md A$ and $\md{\Gamma_{\m}(L)}$ are noetherian over $\comp R$,
the equivalence \eqref{prop100308b7}$\iff$\eqref{prop100308b8}
is standard; see~\cite[Proposition~1.2.3]{bruns:cmr}.
\end{proof}

As with Theorem~\ref{prop100320a}, the next result simplifies when $L$ is noetherian.
Also, see Remark~\ref{disc100319c} for some perspective on the module $\md{\Gamma_{\m}(\md L)}$.

\begin{cor}\label{prop100526a}
Let $A$ be a non-zero artinian $R$-module, and let $L$ be an $R$-module
such that $\beta^R_0(L)<\infty$. The following conditions are equivalent:
\begin{enumerate}[\rm(i)]
\item \label{prop100526a1}
$\Otimes{A}{L}=0$;
\item \label{prop100526a2}
$\ann_{\comp R}(\tors{m}{\md{L}})A=A$; 
\item \label{prop100526a7}
there is an element 
$x\in\ann_{\comp R}(\tors{m}{\md{L}})$ such that $xA=A$; 
\item \label{prop100526a3}
$\ann_{\comp R}(\tors{m}{\md{L}})$ contains a non-zero-divisor for $\md A$; and
\item \label{prop100526a4}
$\att_{\comp R}(A)\cap\supp_{\comp R}(\md{\Gamma_{\m}(\md L)})=\emptyset$.
\end{enumerate}
\end{cor}

\begin{proof}
For an artinian $R$-module $A'$, one has $\att_{\comp R}(A')=\emptyset$ if and only if $A'=0$
by Lemma~\ref{lem100215a}.
 Thus, Proposition~\ref{prop100416d} explains the equivalence
\eqref{prop100526a1}$\iff$\eqref{prop100526a4};
see~\cite[Corollary 2.3]{ooishi:mdwm}.
Since one has $\Otimes{A}{L}=0$ if and only if
$\md{(\Otimes{A}{L})}=0$,
the isomorphism
$\md{(\Otimes{A}{L})}\cong\Hom{A}{\md{L}}$
from Remark~\ref{disc100602b}
in conjunction with Theorem~\ref{prop100320a}
shows that the conditions~\eqref{prop100526a1}--\eqref{prop100526a3}
are equivalent.
\end{proof}

\subsection*{Depth and Vanishing}

\begin{prop}\label{prop100308a}
Let $A$ and $L$ be $R$-modules such that $A$ is artinian.
Then $\ext^i_R(A,L)=0$ for all $i<\depth_R(L)$.
\end{prop}

\begin{proof}
Let $J$ be a minimal $R$-injective resolution of $L$,
and let $i<\depth_R(L)$.
It follows that $\Ext{i}{k}{L}=0$, that is $\mu^i_R(L)=0$, so
the module $E$ does not appear as a summand
of $J^{i}$. As in the proof of Theorem~\ref{thm100308a}, 
this implies that $\Hom AJ^{i}=0$, so $\ext^i_R(A,L)=0$.
\end{proof}

The next example shows that, in 
Proposition~\ref{prop100308a} one may have
$\ext^i_R(A,L)=0$ when $i=\depth_R(L)$.
See also equation~\eqref{prop100616a1}.

\begin{ex} \label{ex100215a}
Assume that $\depth(R)\geq 1$.
Then $\m E=E$ by Lemma~\ref{lem100213c}\eqref{lem100213c1}, so Lemma~\ref{lem100312b} implies that
$$\Ext 0Ek\cong\Hom Ek\cong\Hom{E/\m E}{k}=0$$
even though $\depth_R(k)=0$.
\end{ex}

\begin{prop}\label{prop100419b}
Let $A$ and $L$ be   $R$-modules such that $A$ is artinian.
Then for all $i<\depth_R(\md L)$ one has $\Tor{i}{A}{L}=0$.
\end{prop}

\begin{proof}
When $i<\depth_R(\md{L})$,
one has
$\md{\Tor{i}{A}{L}}\cong\Ext{i}{A}{\md{L}}=0$
by Remark~\ref{disc100602b} and Proposition~\ref{prop100308a},
so $\Tor{i}{A}{L}=0$.
\end{proof}

\begin{thm}\label{prop100616a}
Let $A$ and $A'$ be artinian $R$-modules,
and let $N$ and $N'$ be noetherian $R$-modules. Then one has
\begin{align}
\label{prop100616a1}
\depth_{\comp R}(\ann_{\comp R}(A');\md{A})
&=\inf\{i\geq 0\mid\Ext{i}{A}{A'}\neq 0\}
\\
\label{prop100616a2}
\depth_{R}(\ann_{R}(N');\md{A})
&=\inf\{i\geq 0\mid\Ext{i}{A}{\md{N'}}\neq 0\}\\
\label{prop100616a3}
\depth_{R}(\ann_{R}(N');N)
&=\inf\{i\geq 0\mid\Ext{i}{\md{N}}{\md{N'}}\neq 0\}.
\end{align}
\end{thm}

\begin{proof}
We verify equation~\eqref{prop100616a1} first.
For each index $i$, Theorem~\ref{prop100317a} implies that
\begin{align*}
\Ext{i}{A}{A'}
&\cong\Ext[\comp R]{i}{\md{A'}}{\md{A}}.
\end{align*}
Since $\md{A}$ and $\md{A'}$ are noetherian over $\comp R$,
this explains the first equality below:
\begin{align*}
\inf\{i\geq 0\mid\Ext{i}{A}{A'}\neq 0\}
&=\depth_{\comp R}(\ann_{\comp R}(\md{A'});\md{A})
=\depth_{\comp R}(\ann_{\comp R}(A');\md{A}).
\end{align*}
The second equality is standard
since $\md{A'}=\Hom[\comp R]{A'}{E}$ by Lemma~\ref{lem100206c1}\eqref{lem100206c1b}.

Next, we verify equation~\eqref{prop100616a2}.
Since $\md{N'}$ is artinian, equation~\eqref{prop100616a1} shows that we need only
verify that 
\begin{equation}
\label{prop100616a4}
\depth_{\comp R}(\ann_{\comp R}(\md{N'});\md{A})=\depth_{R}(\ann_{R}(N');\md{A}).
\end{equation}
For this, we compute as follows:
\begin{align*}
\Otimes{\comp R}{N'}
&\stackrel{(1)}{\cong}\Hom[\comp R]{\Hom[\comp R]{\Otimes{\comp R}{N'}}{E}}{E}
\stackrel{(2)}{\cong}\Hom[\comp R]{\md{N'}}{E}.
\end{align*}
Step (1) follows from the fact that $\Otimes{\comp R}{N'}$ is noetherian
(hence, Matlis reflexive) over $\comp R$,
and step (2)  is from Hom-tensor adjointness.
This explains step (4) below:
\begin{align*}
\ann_{\comp R}(\md{N'})
&\stackrel{(3)}{=}\ann_{\comp R}(\Hom[\comp R]{\md{N'}}{E})
\stackrel{(4)}{=}\ann_{\comp R}(\Otimes{\comp R}{N'})
\stackrel{(5)}{=}\ann_{R}(N')\comp R.
\end{align*}
Steps (3) and (5) are standard. This explains  step (6) in the next sequence:
\begin{align*}
\depth_{\comp R}(\ann_{\comp R}(\md{N'});\md{A})
&\stackrel{(6)}{=}\depth_{\comp R}(\ann_{R}(N')\comp R;\md{A})
\stackrel{(7)}{=}\depth_{R}(\ann_{R}(N');\md{A}).
\end{align*}
Step (7) is explained by the following, where step (8) is standard, and step (9) is a consequence
of Hom-tensor adjointness:
\begin{align*}
\Ext[\comp R]{i}{\comp R/\ann_{R}(N')\comp R}{\md{A}}
&\stackrel{(8)}{\cong}\Ext[\comp R]{i}{\Otimes{\comp R}{(R/\ann_{R}(N'))}}{\md{A}}\\
&\stackrel{(9)}{\cong}\Ext{i}{R/\ann_{R}(N')}{\md{A}}.
\end{align*}
This establishes equation~\eqref{prop100616a4} and thus equation~\eqref{prop100616a2}.

Equation~\eqref{prop100616a3} 
follows from~\eqref{prop100616a2} 
because we have
\begin{equation*}
\depth_{R}(\ann_{R}(N');\mdd{N})
=
\width_{R}(\ann_{R}(N');\md{N})
=
\depth_{R}(\ann_{R}(N');N)
\end{equation*}
by Lemma~\ref{lem100213c}\eqref{lem100213c4}.
\end{proof}

\begin{cor}\label{prop100616b}
Let $A$ and $A'$ be artinian $R$-modules,
and let $N$ and $N'$ be noetherian $R$-modules. Then 
\begin{align}
\label{prop100616b1}
\depth_{\comp R}(\ann_{\comp R}(A');\md{A})
&=\inf\{i\geq 0\mid\Tor{i}{A}{\md{A'}}\neq 0\}
\\
\label{prop100616b2}
\depth_{R}(\ann_{R}(N');\md{A})
&=\inf\{i\geq 0\mid\Tor{i}{A}{N'}\neq 0\}\\
\label{prop100616b3}
\depth_{R}(\ann_{R}(N');N)
&=\inf\{i\geq 0\mid\Tor{i}{\md{N}}{N'}\neq 0\}.
\end{align}
\end{cor}

\begin{proof}
We verify equation~\eqref{prop100616b1}; the others are verified similarly.

Since $\Ext{i}{A}{A'}\neq 0$ if and only if $\Hom[\comp R]{\Ext{i}{A}{A'}}{E}\neq 0$,
the isomorphism
$\Hom[\comp R]{\Ext{i}{A}{A'}}{E}\cong\Tor{i}{A}{\md{A'}}$
from
Corollary~\ref{cor100602a}
shows that
$$\inf\{i\geq 0\mid\Ext{i}{A}{A'}\neq 0\}=\inf\{i\geq 0\mid\Tor{i}{A}{\md{A'}}\neq 0\}.$$
Thus equation~\eqref{prop100616b1}
follows from~\eqref{prop100616a1}.
\end{proof}

\section{Examples}\label{sec100420a}

This section contains some explicit computations of Ext and Tor
for the classes of modules discussed in this paper.
Our first example
shows that $\Ext{i}{A}{A'}$ need not be mini-max over $R$.

\begin{ex} \label{ex100421a}
Let $k$ be a field, and set $R=k[X_1,\ldots,X_d]_{(X_1,\ldots,X_d)}$.
We show that $\Hom{E}{E}\cong \comp R$ is not mini-max over $R$.
Note that $R$ is countably generated over $k$, and 
$\comp R\cong k[\![X_1,\ldots,X_d]\!]$ is not countably generated over $k$.
So, $\comp R$ is not countably generated over $R$.
Also, every artinian $R$-module $A$ is 
a countable union of the finite length submodules $(0:_A\m^n)$,
so $A$ is
countably generated.
It follows that every mini-max $R$-module
is also countably generated.
Since $\comp R$ is not countably generated, it 
is not mini-max over $R$.
\end{ex}

Our next example describes $\Ext{i}{A}{A'}$  for some special cases.

\begin{ex}\label{ex100419a}
Assume that $\depth(R)\geq 1$,
and let $A$ be an artinian $R$-module.
Let $x\in\m$ be an $R$-regular element.
The map $E\xra{x}E$ is surjective since $E$ is divisible, and the kernel $(0:_Ex)$ is artinian,
being a submodule of $E$.
Using the injective resolution
$0\to  E\xra{x} E\to 0$ for $(0:_Ex)$,
one can check that
$$\Ext{i}{A}{(0:_Ex)}\cong\begin{cases}
(0:_{\md A} x) & \text{if $i=0$} \\
\md A/x\md A & \text{if $i=1$} \\
0 & \text{if $i\neq0,1$.} \end{cases}$$
For instance, in the case  $A=(0:_Ex)$,  the isomorphism
$\md{(0:_Ex)}\cong \comp R/x\comp R$ implies 
$$\Ext{i}{(0:_Ex)}{(0:_Ex)}\cong\begin{cases}
\comp R/x\comp R & \text{if $i=0,1$} \\
0 & \text{if $i\neq0,1$.} \end{cases}$$
On the other hand, if $x,y$ is an $R$-regular sequence, then
$\md{(0:_Ey)}\cong \comp R/y\comp R$; it follows that $x$ is $\md{(0:_Ey)}$-regular, so one has
$$\Ext{i}{(0:_Ey)}{(0:_Ex)}\cong\begin{cases}
\comp R/(x,y)\comp R & \text{if $i=1$} \\
0 & \text{if $i\neq 1$.} \end{cases}$$
\end{ex}

The next example shows that 
$\Ext{i}{A}{N}$ need not be  mini-max over $R$.

\begin{ex}\label{ex14}
Assume that $R$ is  Cohen-Macaulay with 
$d=\dim(R)$, and let $A$ be an artinian $R$-module.
Assume that $R$ admits a dualizing (i.e., canonical) module $D$.
(For instance, this is so when $R$ is Gorenstein, in which case $D=R$.)
A minimal injective resolution of $D$ has the form
$$J=\quad \textstyle 0\to\coprod_{\Ht(\p)=0}E_R(R/\p)\to\cdots\to\coprod_{\Ht(\p)=d-1}E_R(R/\p)\to E\to 0.$$
In particular, we have 
$\Gamma_{\m}(J)=(0\to 0\to 0\to\cdots\to 0\to E\to 0)$
where the copy of $E$ occurs in  degree $d$.
Since $\Hom{A}{J}\cong\Hom{A}{\Gamma_{\m}(J)}$, it follows that
$$\Ext{i}{A}{D}\cong\begin{cases}
\md A & \text{if $i=d$} \\
0 & \text{if $i\neq d$.} \end{cases}$$
Assume that $d\geq 1$, and let $x\in\m$ be an $R$-regular element.
It follows that the map $D\xra{x}D$ is injective, and the cokernel $D/xD$ is noetherian.
Consider the exact sequence 
$0\to D\xra{x} D\to D/xD\to 0$.
The long exact sequence associated to $\Ext{i}{A}{-}$ shows that
$$\Ext{i}{A}{D/xD}\cong\begin{cases}
(0:_{\md A}x) & \text{if $i=d-1$} \\
\md A/x\md A & \text{if $i=d$} \\
0 & \text{if $i\neq d-1,d$.} \end{cases}$$
As in Example~\ref{ex100419a}, we have
$\md{(0:_Ex)}\cong \comp R/x\comp R$ and
$$\Ext{i}{(0:_Ex)}{D/xD}\cong\begin{cases}
\comp R/x\comp R & \text{if $i=d-1,d$} \\
0 & \text{if $i\neq d-1,d$.} \end{cases}$$
Also, if $x,y$ is an $R$-regular sequence, then
$\md{(0:_Ey)}\cong \comp R/y\comp R$ and
$$
\Ext{i}{(0:_Ey)}{D/xD}\cong\begin{cases}
\comp R/(x,y)\comp R & \text{if $i=d$} \\
0 & \text{if $i\neq d$.} \end{cases}$$
\end{ex}

Next, we show that $\Tor{i}{A}{A'}$ need not be noetherian over $R$ or $\comp R$.

\begin{ex}\label{ex100420a}
Assume that $R$ is  Gorenstein and complete with 
$d=\dim(R)$.
(Hence $D=R$ is a dualizing $R$-module.)
Given two artinian $R$-modules $A$ and $A'$,
Theorem~\ref{thm100320b} implies that $\Tor{i}{A}{A'}$ is artinian, hence Matlis reflexive
for each index $i$, since $R$ is complete. 
This explains the first isomorphism below, and 
Remark~\ref{disc100602b} provides the second isomorphism:
\begin{align*}
\Tor{i}{A}{E}
&\cong\mdd{\Tor{i}{A}{E}}
\cong\md{\Ext{i}{A}{\md E}}
\cong\md{\Ext iAR}
\cong\begin{cases}
A & \text{if $i=d$} \\
0 & \text{if $i\neq d$.} \end{cases}
\end{align*}
Example~\ref{ex14} explains the fourth isomorphism.
Assume that $d\geq 1$, and let $x\in\m$ be an $R$-regular element.
Then $\md{(0:_Ex)}\cong R/xR$, so Example~\ref{ex14} implies that
\begin{gather*}\Tor{i}{A}{(0:_Ex)}
\cong\md{\Ext{i}{A}{\md{(0:_Ex)}}}\cong\begin{cases}
A/xA & \text{if $i=d-1$} \\
(0:_Ax) & \text{if $i=d$} \\
0 & \text{if $i\neq d-1,d$} \end{cases}
\\
\Tor{i}{(0:_Ex)}{(0:_Ex)}
\cong\begin{cases}
(0:_Ex) & \text{if $i=d-1,d$} \\
0 & \text{if $i\neq d-1,d$.} \end{cases}
\end{gather*}
On the other hand, if $x,y$ is an $R$-regular sequence, then
\begin{align*}\Tor{i}{(0:_Ey)}{(0:_Ex)}
&
\cong\begin{cases}
\md{(R/(x,y)R)}\cong E_{R/(x,y)R}(k) & \text{if $i=d$} \\
0 & \text{if $i\neq d$.} \end{cases}
\end{align*}
\end{ex}

Lastly, we provide an explicit computation of $\Otimes{E}{E}$.

\begin{ex}\label{ex100420b}
Let $k$ be a field and set $R=k[\![X,Y]\!]/(XY,Y^2)$.
This is the completion of the multi-graded ring
$R'=k[X,Y]/(XY,Y^2)$ with homogeneous maximal ideal $\m'=(X,Y)R'$.
The multi-graded structure on $R'$ is represented in the following diagram:
$$
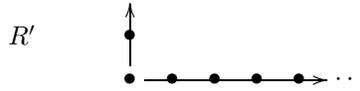
\xymatrix@=.5em{
&&&  &  &  &  &  \\
R'&&& \bullet &  &  &  &  \\
&&& \bullet \ar[uu]\ar[rrrrr] & \bullet & \bullet & \bullet & \bullet & \cdots}$$
where each bullet represents the corresponding monomial in $R'$.
It follows that 
$E\cong E_{R'}(k)\cong k[X^{-1}]\oplus k Y^{-1}$
with graded module structure given by the formulas
\begin{align*}
X\cdot 1&=0
&X\cdot X^{-n}&=X^{1-n}
&X\cdot Y^{-1}&=0
\\
Y\cdot 1&=0
&Y\cdot Y^{-1}&=1
&Y\cdot X^{-n}&=0
\end{align*}
for  $n\geq 1$.
Using this grading, one can show that
$\m E=\m'E\cong k[X^{-1}]$ and $\m^2E=\m E$. These modules are represented in the next diagrams:
$$
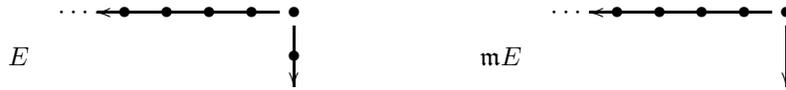
\xymatrix@=.5em{
&\cdots & \bullet & \bullet & \bullet & \bullet & \bullet \ar[lllll] \ar[dd] &  \\
E&& & & & & \bullet &  \\
&&&&&&&&}
\qquad
\qquad
\xymatrix@=.5em{
&\cdots & \bullet & \bullet & \bullet & \bullet & \bullet \ar[lllll] \ar[dd] &  \\
\m E&& & & & &  &  \\
&& & & & & &}$$
It follows that $E/\m E\cong k$, so Lemma~\ref{lemma1} implies that
$$\Otimes{E}{E}\cong \Otimes{(E/\m E)}{(E/\m E)} \cong\Otimes{k}{k}\cong k.$$
A similar computation shows the following: Fix positive integers $a,b,c$ such that $c>b$, 
and consider the ring
$S=k[\![X,Y]\!]/(X^aY^b,Y^c)$ with maximal ideal $\n$ and $E_S=E_{S}(k)$.
Then $\n^{c-b}E_S=\n^{c-b+1}E_S$
and we get the following:
\begin{gather*}
E_S/\n^{c-b}E_S\cong S/(X^a,Y^{c-b})S\cong k[X,Y]/(X^a,Y^{c-b})\\
\Otimes[S]{E_S}{E_S}\cong \Otimes[S]{(E_S/\n^{c-b} E_S)}{(E_S/\n^{c-b} E_S)} \cong S/(X^a,Y^{c-b})S.
\end{gather*}
\end{ex}

\section*{Acknowledgments}
We are grateful to 
Luchezar Avramov,
C\u{a}t\u{a}lin Ciuperc\u{a},  
Edgar Enochs,
Srikanth Iyengar, and
Roger Wiegand
for useful feedback about this research.

\providecommand{\bysame}{\leavevmode\hbox to3em{\hrulefill}\thinspace}
\providecommand{\MR}{\relax\ifhmode\unskip\space\fi MR }
\providecommand{\MRhref}[2]{%
  \href{http://www.ams.org/mathscinet-getitem?mr=#1}{#2}
}
\providecommand{\href}[2]{#2}

\end{document}